\documentclass[12pt]{article}
\usepackage{amsmath, amsthm}
\usepackage{amsfonts}
\usepackage{amssymb}
\usepackage[usenames,dvipsnames]{color}
\usepackage{graphicx}
\usepackage{hyperref}
\usepackage{bm}
\newtheorem{proposition}{Proposition}
\newtheorem{theorem}{Theorem}
\newtheorem{lemma}{Lemma}
\newtheorem{cor}{Corollary}

\newcommand{\RR}{\mathbb{R}}

\newcommand{\NN}{\mathbb{N}}
\newcommand{\CC}{\mathbb{C}}

\newcommand{\cdotsc}{,\dotsc,}

\newcommand{\stproca}[1]{\left(#1\right)_{t \ge 0}}
\newcommand{\stproc}[1]{\stproca{#1_t}}

\DeclareMathOperator{\rRe}{Re}

\renewcommand{\Re}{\rRe}

\newcommand{\iu}{\mathrm{i}} 

\newenvironment{eqnarr}{\begin{IEEEeqnarray}{rCl}}{\end{IEEEeqnarray}\ignorespacesafterend}

\renewcommand{\eqref}[1]{\hyperref[#1]{(\ref*{#1})}}

\newcommand{\rhohat}{\hat{\rho}}

\newcommand{\GGt}{(\mathcal{G}_t)_{t \ge 0}}

\DeclareMathOperator{\diag}{diag}
\newcommand{\Had}{\circ}

\newcommand{\dd}{\mathrm{d}}

\newcommand{\for}{\qquad}

\newcommand{\define}{\emph}

\newcommand{\abs}[1]{\lvert #1 \rvert}

    \def\beq{\begin{eqnarr}}
    \def\eeq{\end{eqnarr}}
    \def\beqq{\begin{eqnarray*}} 
    \def\eeqq{\end{eqnarray*}} 

    \def\d{{\textnormal d}}
\newtheorem{remark}{Remark}

\newcommand*{\pref}[1]{\hyperref[#1]{(\ref*{#1})}}
\newcommand*{\refpref}[2]{\hyperref[#2]{\ref*{#1}(\ref*{#2})}}

\title{Deep  factorisation of the stable process}

\author{Andreas E. Kyprianou\thanks{Department of Mathematical Sciences, University of Bath, Claverton Down, Bath, BA2 7AY, UK. Email: \texttt{a.kyprianou@bath.ac.uk}} 
}

\begin{document}

\maketitle

\begin{abstract}

The Lamperti--Kiu transformation for real-valued self-similar Markov processes (rssMp) states that, associated to each rssMp via a space-time transformation, there is a Markov additive process (MAP). In the case that the rssMp is taken to be  an $\alpha$-stable process with $\alpha\in(0,2)$, \cite{CPR} and \cite{T_0} have computed explicitly the characteristics of the matrix exponent of the semi-group of the embedded MAP, which we henceforth refer to as the {\it Lamperti-stable MAP}. Specifically, the matrix exponent of the Lamperti-stable MAP's transition semi-group can be written in a compact form using only gamma functions. 

Just as with L\'evy processes, there exists a factorisation of the (matrix) exponents of MAPs, with each of the two factors uniquely characterising the ascending and descending ladder processes, which themselves are again MAPs. 
 To the author's knowledge, not a single  example of such a factorisation currently exists in the literature. 
In this article we provide a completely explicit Wiener--Hopf factorisation for the Lamperti-stable MAP. 

The main value and novelty of exploring the matrix Wiener--Hopf factorisation of the underlying MAP comes about through style of the computational approach. Understanding the fluctuation theory of the underlying MAP offers new insight into different ways of analysing stable processes. Indeed, we obtain  new space-time invariance properties of stable processes, as well as demonstrating examples how new fluctuation identities for stable processes can be developed as a consequence of the reasoning in deriving the matrix Wiener--Hopf  factors. The methodology in this paper has already lead to new applications in the forthcoming work of \cite{Deep2} and \cite{KRS}.


\medskip


\medskip

\end{abstract}

\section{Introduction}\label{c5_introduction}

Let $X: = (X_t)_{t\geq 0}$ be a one-dimensional L\'evy process,
starting from zero,
with law $\mathbb{P}$. The L\'evy--Khintchine formula
states that, for all 
$\theta\in \mathbb{R}$,
the characteristic exponent
$\Psi(\theta) : = 
-\log \mathbb{E} ({\rm e}^{\iu\theta X_1})$ satisfies 
\begin{equation}
\Psi(\theta)
  = \iu a\theta 
  + \frac{1}{2}\sigma^2\theta^2 
  + \int_{\mathbb{R}} (1 - {\rm e}^{\iu\theta x} + \iu\theta x\mathbf{1}_{(|x|\leq 1)})\Pi({\rm d} x),\qquad \theta\in\mathbb{R},
\label{charexp}
\end{equation}
where $a\in\mathbb{R}$, $\sigma\geq 0$ and $\Pi$ is a measure
(the \textit{L\'evy measure}) concentrated on
$\mathbb{R}\setminus\{0\}$
such that $\int_{\mathbb{R}}(1\wedge x^2)\Pi({\rm d} x)<\infty$.
When analytical extension is possible, we refer to $\psi(z): = -\Psi(-\iu z)$ as the Laplace exponent.

The process $(X,\mathbb{P})$ is
said to be a \textit{strictly $\alpha$-stable process} (henceforth just written `stable process')
if it is an unkilled L\'evy process which
also satisfies the \textit{scaling property}: under $\mathbb{P}$,
for every $c > 0$,
the process $(cX_{t c^{-\alpha}})_{t \ge 0}$
has the same law as $X$.
It is known that $\alpha$ necessarily belongs to $(0,2]$, and the case $\alpha = 2$
corresponds to Brownian motion, which we exclude.
The L\'evy-Khintchine representation of such a process
is as follows:
$\sigma = 0$, 
$\Pi$ is absolutely continuous with density given by 
\[ 
 \pi(x): =  c_+ x^{-(\alpha+1)} \mathbf{1}_{(x > 0)} + c_- \abs{x}^{-(\alpha+1)} \mathbf{1}_{(x < 0)},
\qquad  x \in \mathbb{R},
\]
where $c_+,\, c_- \ge 0$, and $a = (c_+-c_-)/(\alpha-1)$.

The process $X$ has the
characteristic exponent
\begin{equation}\label{e:stable CE}
  \Psi(\theta) =
  c\abs{\theta}^\alpha
  (1  - \iu\beta\tan\tfrac{\pi\alpha}{2}\text{sgn}(\theta)), \qquad   \theta\in\mathbb{R},
\end{equation}
where $\beta = (c_+- c_-)/(c_+ + c_-)$ and 
$c = - (c_++c_-)\Gamma(-\alpha)\cos (\pi\alpha/2)$. Self-similarity dictates that we must necessarily have $\beta = 0$ when $\alpha = 1$, which is to say that the process is symmetric. For more details,
see \cite[Theorems 14.10 and 14.15]{Sato}. 

For consistency with the literature that we shall appeal to in this article,
we shall always parametrise our $\alpha$-stable process such that 
\[ c_+ = \Gamma(\alpha+1) \frac{\sin(\pi \alpha \rho)}{\pi} \quad \text{and} \quad
  c_- = \Gamma(\alpha+1) \frac{\sin(\pi \alpha \hat\rho)}{\pi },
  \]
where
$\rho = \mathbb{P}(X_t \ge 0)$ is the positivity parameter, 
and $\hat\rho = 1-\rho$.  In that case, the constant $c$ simplifies to just $c = \cos (\pi\alpha(\rho - 1/2))$. Moreover,   we may also identify the exponent as taking the form 
\begin{equation}\label{Psi_alpha_rho_parameterization}
\Psi(\theta) = |\theta|^\alpha ({\rm e}^{\pi\iu\alpha(\frac{1}{2} -\rho)} \mathbf{1}_{(\theta>0)} + {\rm e}^{-\pi\iu\alpha (\frac{1}{2} - \rho)}\mathbf{1}_{(\theta<0)}), \qquad \theta\in\mathbb{R}.
\end{equation}

With this normalisation, we take the point of view that the class of stable processes is
parametrised by $\alpha$ and $\rho$; the reader will note that
all the quantities above can be written in terms of these parameters.
We shall restrict ourselves a little further within this class
by excluding the possibility of having only one-sided jumps. In particular, this rules out the possibility that $X$ is a subordinator or the negative of a subordinator, which occurs when $\alpha\in(0,1)$ and either $\rho =1$ or $0$. In the case of a subordinator, for future reference, we note that the characteristic exponent takes the form 
$\Psi(\theta) = 
(-\iu\theta)^\alpha
$, $\theta\in\mathbb{R}$, which is the analytic extension of the Bernstein function $\lambda \mapsto \lambda^\alpha$, $\lambda \geq 0$.

A fascinating theoretical feature of all characteristic exponents of L\'evy processes is that they can always be written in terms of the so-called {\it Wiener--Hopf factors}. That is to say, for a given characteristic exponent of a L\'evy process, $\Psi$, there exist unique Bernstein functions, $\kappa$ and $\hat\kappa$ such that, up to a multiplicative constant,
\begin{equation}\label{LWHF}
\Psi(\theta) = \hat\kappa(\iu\theta)\kappa(-\iu\theta), \qquad \theta\in\mathbb{R}.
\end{equation}
As Bernstein functions, $\kappa$ and $\hat\kappa$ can be seen as the exponents of (killed) subordinators. The probabilistic significance of these subordinators, known as the ascending and descending ladder height processes respectively, is that their range corresponds precisely to the range of the running maximum of $X$ and of $-X$ respectively. In this sense, they play an important role in understanding the path fluctuations of the underlying L\'evy processes. In particular, a rich history of literature has shown their fundamental significance in the understanding of a variety first passage problems; see for example their extensive use the the development of fluctuation theory  of L\'evy processes in the texts \cite{BertoinLP} and \cite{Kyp} as well as \cite{Bingham}. 

In the case of stable processes, the Wiener--Hopf factorisation takes a relatively straightforward form. Indeed, it is straightforward to argue that the ascending and descending ladder processes must necessarily be stable subordinators. One is therefore forced to take (up to a multiplicative constants) $\kappa(\lambda)=\lambda^{\alpha_1}$, $\lambda\geq 0$, and $\hat \kappa(\lambda)=\lambda^{\alpha_2}$, $\lambda\geq 0$, for some $\alpha_1, \alpha_2 \in (0,1)$. Comparing (\ref{LWHF}) with \eqref{Psi_alpha_rho_parameterization}, we must choose the parameters $\alpha_1$ and $\alpha_2$ such that, for example, when $z>0$,  
\begin{equation}
z^{\alpha} {\rm e}^{\pi \iu\alpha (\frac{1}{2}-\rho)}=z^{\alpha_1} {\rm e}^{-\frac{1}{2}{\pi \iu\alpha_1}} \times 
z^{\alpha_2} {\rm e}^{\frac{1}{2}{\pi \iu\alpha_2}}.
\label{WHFSproof}
\end{equation}
Matching radial and angular parts, we find that 
\begin{eqnarray}
\begin{cases}
\alpha_1+\alpha_2=\alpha, \\
\alpha_1-\alpha_2=-\alpha(1-2\rho),
\end{cases}
\end{eqnarray}
which gives us $\alpha_1=\alpha \rho$ and $\alpha_2=\alpha\hat\rho$. As we have assumed that $X$ does not have monotone paths, it is necessarily the case   that $0< \alpha\rho\leq 1$ and $0<\alpha\hat\rho\leq 1$. Note also that when $\alpha\rho=1$, the ascending ladder height process is a pure linear drift. In that case, the range of the maximum process $\overline{X}$ is $[0,\infty)$. This can only happen when  $X$ is spectrally negative which  has been ruled out by assumption in the introduction. Similarly the case that $\alpha\hat\rho=1$ corresponds to spectral positivity which has also been  ruled out by assumption.
In conclusion, 
\[
\kappa(\lambda) = \lambda^{\alpha\rho}\text{ and }\hat\kappa(\lambda) = \lambda^{\alpha\hat\rho}, \qquad \lambda \geq 0
\]
where $0<\alpha\rho,\alpha\hat\rho<1$. 

This discussion also helps us justify that, taking account of all the special cases of stable processes that we have chosen to exclude, the set of admissible parameters  we are left to work with is
\[
  \bigl\{ (\alpha,\rho) : \alpha \in (0,2), \, \rho \in (1-1/\alpha, 1/\alpha)
 \text{ and } \rho = {1}/{2} \text{ if } \alpha = 1
   \bigr\}.
\]


In this article, we expose a second Wiener--Hopf factorisation which is `deeply' embedded within the stable processes through its so-called Lamperti--Kiu representation. Like the factorisation (\ref{WHFSproof}), the `deep factorisation' we will present has value in that it informs us about the fluctuations of the stable process.  As we shall see, this gives us access to an array of new results for stable processes, as well as a methodology for obtaining even more than those presented in this paper.

The Lampert--Kiu representation is a pathwise decomposition that holds more generally for any real-valued self-similar Markov processes (rssMp)  and shows that any such process can be written as a space-time changed Markov additive process (MAP). In a similar spirit to (\ref{charexp}), the semi-group of a MAP can be characterised via an exponent, albeit that it now takes the form of a complex-valued matrix function. Moreover,  just as with L\'evy processes, there exists a factorisation of the aforesaid matrix exponent.
Not a single concrete example of such a factorisation currently exists in the literature for such MAPs to the author's knowledge. 
Our main objective here is to provide a completely explicit Wiener--Hopf factorisation for the MAP that underlies the stable process via the Lamerti--Kiu transform: the so-called {\it Lamperti-stable MAP}.

Our approach appeals to three main techniques. First, information about the the respective individual entries in the matrix factors can be gleaned using asymptotic Markov additive renewal theory in the setting of excursion theory for MAPs. Second, the quantities that are identified in that way can be related to complex first passage problems for stable processes 
Third, the aforesaid first entry problems can be simplified by appealing to a version of the Riesz--Bogdan--Zak transform, which relates the  mapping of the path of a stable process via a Kelvin transform, together with an endogenous time change, to a Doob $h$-transform of the stable process. For $\alpha\in(1,2)$, this Doob $h$-transform corresponds to conditioning the stable process to avoid the origin, as explored in \cite{CPR}. For $\alpha\in(0,1)$ it corresponds to conditioning the stable process to being absorbed at the origin. Finally, for $\alpha=1$ there is, in effect, no $h$-transform as  $h\equiv 1$.

\bigskip

The remainder of this article is structured as follows. In the next section we explain the nature of the  Lamperti--Kiu representation for pssMps, due to \cite{CPR}, and its relation to MAPs. In particular we give the example of the of the Lamperti-stable MAP, also due to \cite{CPR}.
With this in hand, we are able to state our `deep' Wiener--Hopf factorisation result for the stable process. 
In section \ref{sectBZ} we discuss the Riesz--Bogdan--Zak transform.
In section \ref{ascending} we compute the first matrix factor of the Wiener--Hopf factorisation for the Lampert-stable MAP, which is the analogue of the contribution from the exponent of the ascending ladder height process in the Wiener--Hopf factorisation for L\'evy processes. In section \ref{updual}, we compute the second  Wiener--Hopf matrix factor, which is the analogue of the contribution from the ascending ladder height process of the dual in the Wiener--Hopf factorisation for L\'evy processes. Finally, in Section \ref{cramer}, we outline  how some Cram\'er-type asymptotics and explicit identities for the Lamperti-stable MAP and stable processes can be obtained from the methods that underly the deep factorisation. 

\section{MAPs and the Lamperti--Kiu transform}

This section is laid out as follows. We
devote the first two subsections to a discussion of Markov
additive processes and 
real self-similar Markov processes via the Lamperti--Kiu
representation. Finally, in the last
subsection, give our main result, the deep Wiener--Hopf factorisation of the stable process. 

\subsection{Markov additive processes}
\label{ss:MAP}

Let $E$ be a finite state space and $\GGt$ a standard
filtration. 
A c\`adl\`ag process $(\xi,J)$ in $\mathbb{R} \times E$
with law $\mathbf{P}$ is called a
\define{Markov additive process (MAP)} with respect to $\GGt$
if $(J(t))_{t \ge 0}$ is a continuous-time Markov chain in $E$, and
the following property is satisfied,
for any $i \in E$, $s,t \ge 0$:
\begin{eqnarray}
\label{e:MAP}
& \text{ given $\{J(t) = i\}$,
the pair $(\xi(t+s)-\xi(t), J(t+s))$ is independent of
$\mathcal{G}_t$,}\notag\\
&\text{ and has the same distribution as $(\xi(s)-\xi(0), J(s))$
given $\{J(0) = i\}$.}
\end{eqnarray}

Aspects of the theory of Markov additive processes
are covered in a number of texts, among them
\cite{Asm-rp1} and \cite{Asm-apq2}.
We will mainly use the notation of
\cite{Iva-thesis},
where it was principally assumed that 
$\xi$ is spectrally negative; the results which
we quote are valid without this hypothesis, however.

Let us introduce some notation.
For $x\in\mathbb{R}$,  write $\mathbf{P}_{x,i} = \mathbf{P}( \cdot \,\vert\, \xi(0) = x, J(0) = i)$.
If $\mu$ is a probability distribution on $E$, we write
$\mathbf{P}_{x, \mu}  
  = \sum_{i \in E} \mu_i \mathbf{P}_{x,i}$. 
We adopt a similar convention
for expectations.

It is well-known that a Markov additive process $(\xi,J)$ also satisfies
\eqref{e:MAP} with $t$ replaced by a stopping time, albeit on the event that the stopping time is finite. The following proposition gives a characterisation of MAPs in terms of a mixture of L\'evy processes, a Markov chain and a family of additional jump distributions;
see \cite[\S XI.2a]{Asm-apq2}, \cite[Proposition 2.5]{Iva-thesis} as well as more classical literature such as \cite{Cinlar1, Cinlar2, AS}.

\begin{proposition}
  The pair $(\xi,J)$ is a Markov additive process if and only if, for each $i,j\in E$, 
  there exist a sequence of iid L\'evy processes
  $(\xi_i^n)_{n \ge 0}$ and  a sequence of iid random variables
  $(U_{i,j}^n)_{n\ge 0}$, independent
  of the chain $J$, such that if $\sigma_0 = 0$
  and $(\sigma_n)_{n \ge 1}$ are the
  jump times of $J$, the process $\xi$ has the representation
  \[ \xi(t) = \mathbf{1}_{(n > 0)}( \xi(\sigma_n -) + U_{J(\sigma_n-), J(\sigma_n)}^n) + \xi_{J(\sigma_n)}^n(t-\sigma_n),
    \for t \in [\sigma_n, \sigma_{n+1}),\, n \ge 0. \]
\end{proposition}
For each $i \in E$, it will be convenient to define,
on the same probability space, $\xi_i$ as a L\'evy process whose distribution is the common
  law of the $\xi_i^n$ processes in the above representation; and similarly, for each $i,j \in E$, define $U_{i,j}$ to
  be a random variable having the common law of the $U_{i,j}^n$ variables.

Henceforth, we confine ourselves
to irreducible (and hence ergodic) Markov chains $J$.
Let the state space $E$ be the finite set $\{1 \cdotsc N\}$, for some $N \in \NN$.
Denote the transition rate matrix of the chain $J$ by
${\boldsymbol{Q}} = (q_{i,j})_{i,j \in E}$.
For each $i \in E$, the Laplace exponent of the L\'evy process $\xi_i$
will be written $\psi_i$. 
For each pair of $i,j \in E$,
define
the Laplace transform $G_{i,j}(z) = \mathbf{E}({\rm e}^{z U_{i,j}})$
 of the jump
distribution $U_{i,j}$,
where this exists. Write ${\boldsymbol{G}}(z)$ for the $N \times N$ matrix
whose $(i,j)$th element is $G_{i,j}(z)$. 
We will adopt the convention that $U_{i,j} = 0$ if
$q_{i,j} = 0$, $i \ne j$, and also set $U_{ii} = 0$ for each $i \in E$.

The multidimensional analogue of the Laplace exponent of a L\'evy process is
provided by the matrix-valued function
\begin{equation}\label{e:MAP F}
 {\bf F}(z) = \diag( \psi_1(z) \cdotsc \psi_N(z))
  + {\boldsymbol{Q}} \Had {\boldsymbol{G}}(z),
\end{equation}
for all $z \in \CC$ where the elements on the right are defined,
where $\Had$ indicates elementwise multiplication, also called
Hadamard multiplication.
It is then known
that
\[ \mathbf{E}_{0,i}( {\rm e}^{z \xi(t)} ; J(t)=j) = \bigl({\rm e}^{\boldsymbol{F}(z) t}\bigr)_{i,j} , \for i,\,j \in E,  t\geq 0,\]
for all $z \in \CC$ where one side of the equality is defined.
For this reason, $\boldsymbol{F}$ is called the \define{matrix exponent} of
the MAP $(\xi, J)$.

Just as is the case with L\'evy processes, the exponents of MAPs are also known to have a Wiener--Hopf factorisation. However, this time, the two factors correspond to the matrix exponent of the  ascending (resp.  descending) ladder processes.  These are  themselves MAPs with trajectories which agree with    the range and state of the modulating chain at times of new maxima (resp. minima).  In order to explain the nature of the Wiener--Hopf factorisation for MAPs, we need to introduce a little more notation.

Associated to the running maximum process $(\sup_{s\leq t}\xi(s))_{t\geq 0}$ is a Markov additive subordinator. That is, a MAP, say $(H^+(t), J^+(t))_{t\geq 0}$, with the property that $H^+ $ is non-decreasing  with the same range as the running maximum. Moreover, its exponent can be identified by  $-\boldsymbol{\kappa}(-z)$, where
\begin{equation}
\boldsymbol\kappa(\lambda) = \text{diag}(\Phi _1(\lambda), \cdots, \Phi _N(\lambda)) - {\boldsymbol{\Lambda}}\circ {\boldsymbol{K}}(\lambda),
\qquad \lambda\geq 0,
\label{MAPBernstein}
\end{equation}
is a matrix analogue of a Bernstein function.
Here, for $i =1,\cdots, N$, $\Phi _i$ are Bernstein functions (exponents of subordinators), ${\boldsymbol{\Lambda}}  = (\Lambda_{i,j})_{i,j\in E}$ is the intensity matrix of $J^+$ and ${\boldsymbol{K}}(\lambda)_{i,j} = {\mathbf E}[{\rm e}^{-\lambda U^+_{i,j}}]$, where $U^+_{i,j}\geq 0$ are the additional discontinuities added to the path of $\xi$ each time the chain $J^+$ switches from $i$ to $j$, and $U^+_{i,i}: = 0$, $i\in E$.
 
We also need to talk about the same quantity but for the dual of $(\xi, J)$. Whilst the dual of a L\'evy process is equal in law to nothing more than its negative, the situation for MAPs is a little more involved. First note that, thanks to irreducibility,  the Markov chain $J$ necessarily has a stationary distribution. We denote it by the vector $\boldsymbol\pi  = (\pi_1, \cdots, \pi_N)$.  The dual process that is the  MAP with probabilities $\hat{\mathbf P}_{x,i}$, $x\in\mathbb{R}$, $i\in E$, whose matrix exponent, when it is defined, is given by, 
\begin{align*}
	\hat{\mathbf E}_{0,i}\big[ {\rm e}^{z {\xi}(t)},J(t)=j\big]=\big( {\rm e}^{\hat{\boldsymbol{F}}(z)t}\big)_{i,j},\quad i,j\in E,
\end{align*}
where 
\[
	\hat{\boldsymbol{F}}(z):=\text{diag}\big(\psi_1(-z),...,\psi_{|E|}(-z)\big)+\hat{{\boldsymbol{Q}}} \circ {\boldsymbol{G}}(-z)^{\rm T}
\]
and $\hat {\boldsymbol{Q}}$ is the intensity matrix of the modulating Markov chain on $E$ with entries given by
\[
\hat{q}_{i,j} = \frac{\pi_j}{\pi_i}q_{j,i}, \qquad i,j\in E.
\]
Note that the latter can also be written $\hat{{\boldsymbol{Q}}} = \boldsymbol{\Delta}_\pi^{-1} {\boldsymbol{Q}}^{\rm T} \boldsymbol{\Delta}_{\boldsymbol\pi}$, where $\boldsymbol{\Delta}_{\boldsymbol\pi} = \text{diag}(\boldsymbol{\pi})$, the matrix with diagonal entries given by $\boldsymbol\pi$ and zeros everywhere else. Hence, when it exists,
\begin{equation}
\hat{\boldsymbol{F}}(z) = \boldsymbol{\Delta}_\pi^{-1}\boldsymbol{F}(-z)^{\rm T}\boldsymbol{\Delta}_\pi, 
\label{dualF}
\end{equation}
showing that 
\begin{equation}
	\pi_i\hat{\mathbf E}_{0,i}\big[ {\rm e}^{z {\xi}(t)},J(t)=j\big]=\pi_j{\mathbf E}_{0,j}\big[ {\rm e}^{-z {\xi}(t)},J(t)=i\big].
	\label{incrementduality}
\end{equation}
At the level of processes, one can understand (\ref{incrementduality}) as saying the following.
\begin{lemma}\label{duality}The time-reversed process $\{ \left(\xi((t-s)-) -\xi(t), J((t-s)-) \right): s\leq t\}$ under $\mathbf{P}_{0,{\boldsymbol\pi}}$ is equal in law to $\{(\xi(s), J(s)) : s\leq t\}$ under $\hat{\mathbf P}_{0,{\boldsymbol\pi}}$. 
\end{lemma} 
 
 We are now ready to state the Wiener--Hopf factorisation for MAPs. Whilst some results in this direction exist in classical literature, see for example Chapter XI of \cite{Asmussen} or Theorem 3.28 of \cite{Kaspi}. Several other references can be cited in this respect, for example  \cite{KP} and \cite{AS}. None of them are in an appropriate form for our purposes. We have lifted the following result from the Appendix of the recent article \cite{P_0}.
 \begin{theorem}\label{WHF}
For $\theta\in \mathbb{R}$, up to an multiplicative constant,
\[
- \boldsymbol{F}(\iu \theta) = \boldsymbol{\Delta}_{\boldsymbol\pi}^{-1}\hat{\boldsymbol{\kappa}}( \iu \theta)^{\rm T}\boldsymbol{\Delta}_{\boldsymbol\pi}\boldsymbol{\kappa}(-{\rm i}\theta) ,
\]
where  $\hat{\boldsymbol{\kappa}}$ plays the role of $\boldsymbol{\kappa}$, but for the dual MAP to $(\xi, J)$.
\end{theorem}

Note that this Theorem is consistent with the Wiener--Hopf factorisation for L\'evy processes \eqref{LWHF} as, in that setting, the dual process is its negative.

\subsection{Real self-similar Markov processes}\label{rssmp}
\newcommand{\CPRE}{\mathcal{E}}
\newcommand{\mT}{\mathcal{T}} 

The structure of real self-similar Markov processes has been investigated
by \cite{Chy-Lam} in the symmetric case, and \cite{CPR} in general.
Here, we give an interpretation of these authors' results
in terms of a two-state Markov additive process. We begin with some
relevant definitions and introductory results.

A \define{real self-similar Markov process} (rssMp) with \define{self-%
similarity index} $\alpha > 0$ is a standard (in the sense of \cite{BG-mppt})
Markov process $X = \stproc{X}$ with probability laws
$(\mathbb{P}_x)_{x \in \RR}$ which satisfies the
\define{scaling property} that for all $x \in \RR \setminus \{0\}$
and $c > 0$,
\[ \text{the law of }(c X_{t c^{-\alpha}})_{t \ge 0}
  \text{ under } \mathbb{P}_x \text{ is } \mathbb{P}_{cx} . \]

In \cite{CPR} the authors confine their attention to processes
in `class \textbf{C.4}'. An rssMp $X$ is in this class if, for
all $x \ne 0$, $\mathbb{P}_x( \exists t > 0: X_t X_{t -} < 0 ) = 1$;
that is, with probability one, the process $X$ changes
sign infinitely often.
Define
\[ \tau^{\{0\}} = \inf\{t \ge 0: X_t = 0 \},\]
the time to absorption at the origin.

Such a process may be identified with a MAP via a deformation of space and
time  which we call the
\define{Lamperti--Kiu representation} of $X$. The following
result is a simple corollary of \cite[Theorem 6]{CPR}.

\begin{proposition}
  Let $X$ be an rssMp in class \textbf{C.4} and fix $x \ne 0$.
  Define the symbol
  \[ [y] = \begin{cases}
             1, & y > 0, \\
             2, & y < 0.
           \end{cases}
  \]
  Then there exists a time-change $\sigma$, adapted to the filtration of $X$,
  such that, under the law $\mathbb{P}_x$, the process 
  \[ (\xi(t),J(t)) = (\log\abs{X_{\sigma(t)}}, [X_{\sigma(t)}]) , \qquad t \ge 0, \]
  is a MAP with state space $E = \{1,2\}$
  under the law $\mathbf{P}_{\log |x|,[x]}$.
  Furthermore, the process $X$ under $\mathbb{P}_x$ has the representation
  \[ X_t =  \exp\bigl( \xi( \varphi(t))
    + \iu \pi (J( \varphi(t)) + 1 )\bigr) , \for 0 \le t < \tau^{\{0\}}, \]
  where $\varphi$ is the inverse of the time-change $\sigma$,
  and may be given by
\begin{equation}\label{e:Lamp time change}
  \varphi(t) = \inf \biggl\{ s > 0 : \int_0^s \exp(\alpha \xi(u))
  \, \dd u > t  \biggr\}, \for t < \tau^{\{0\}},
\end{equation}
such that $(\xi, J)$ has law $\mathbf{P}_{\log x, [x]}$.
\end{proposition}



\subsection{The Lamperti-Stable MAP and deep factorisation}
Now let us return to the case that $X$ is a stable process as described in the introduction, which is also an rssMp.  
In \cite[\S 4.1]{CPR}, the authors calculate the characteristics
of the Lamperti--Kiu representation for $X$ until absorption at the origin; that is, they compute the characteristics of the processes 
$\xi_i$, the jump distributions $U_{i,j}$ and rates the $q_{i,j}$, for $i,j\in \{1,2\}$.
Using this information, and the representation \eqref{e:MAP F},
it was shown in \cite{T_0} that the MAP $( \xi,J)$ has matrix exponent

\begin{equation}
  \boldsymbol{F}(z) =\left[
  \begin{array}{cc}
    - \dfrac{\Gamma(\alpha-z)\Gamma(1+z)}
      {\Gamma(\alpha\hat\rho-z)\Gamma(1-\alpha\hat\rho+ z)}
    & \dfrac{\Gamma(\alpha-z)\Gamma(1+z)}
      {\Gamma(\alpha\hat\rho)\Gamma(1-\alpha\hat\rho)}
    \\
    &\\
    \dfrac{\Gamma(\alpha-z)\Gamma(1+ z)}
      {\Gamma(\alpha\rho)\Gamma(1-\alpha\rho)}
    & - \dfrac{\Gamma(\alpha-z)\Gamma(1+z)}
      {\Gamma(\alpha\rho-z)\Gamma(1-\alpha\rho+z)}
  \end{array} 
  \right],
  \label{MAPHG}
\end{equation}
for $\Re(z)\in(-1,\alpha)$.
In the spirit of \cite[Chapter 13.4]{Kyp} we refer to this process as a {\it Lamperti-stable MAP}.

We should also note that the diagonal terms have entries which are characteristic exponents which belong to the class of so-called {\it hypergeometric L\'evy processes}. Moreover, up to a multiplicative constant, the off-diagonal terms can be shown to be the Laplace transforms of distributions, which possess a density with respect to Lebesgue measure that can be written in terms of the classical hypergeometric $_2\mathcal{F}_1$ function (see Chapter 13 of \cite{Kyp}). The matrix exponent (\ref{MAPHG}) could, in theory, be shown to belong to a bigger family of MAPs which are in some sense a natural generalisation of the class of hypergeometric L\'evy processes (cf. \cite{KP}). Indeed, we shall see other MAPs in the forthcoming analysis which are different to the Lamperti-stable MAP but clearly are close relatives with a common analytic structure.
We shall explore this remark in more detail in future work however.

Our main result, below, gives the explicit factorisation of (\ref{MAPHG}) as predicted by Theorem \ref{WHF}. To our knowledge this is the first time that an example of the Wiener--Hopf factorisation for a MAP has has been detailed explicitly.

\bigskip

We first need to introduce some notation. Of use will be the family of  Bernstein functions
\begin{equation}
\kappa_{q+i,p+j}(\lambda): =
\int_0^\infty (1- {\rm e}^{-\lambda x} )\frac{((q+i)\vee (p+j)-1)}{ (1-{\rm e}^{-x})^{q+i}(1+{\rm e}^{-x})^{p+j} }{\rm e}^{-\alpha x}\d x, \qquad \lambda \geq 0, 
\label{kappas}
\end{equation}
where  $q,p\in\{\alpha\rho,\alpha\rhohat\}$ and $i,j\in\{0,1\}$ such that $q+p = \alpha$ and  $i+j =1$. Note that it is easy to verify that the above expression is indeed a Bernstein function as the associated L\'evy density behaves like either $x^{-\alpha\rho -1}$ or $x^{-\alpha\rhohat-1}$ as $x\downarrow0$ and like ${\rm e}^{-\alpha x}$ as $x\uparrow\infty$. Accordingly, it is also straightforward to verify that the mean value $\kappa'_{q+i,p+j}(0+)$ is finite.

\begin{theorem}
\label{WHFMAPHG} When $\alpha\in(0,1]$,
We have the following two components to the factorisation in Theorem \ref{WHF}.

\bigskip

\noindent (i) Up to a multiplicative constant,
the ascending ladder MAP exponent is given by 
\begin{eqnarray*}
\boldsymbol{\kappa}(\lambda) &=& \scriptsize{
\left[
\begin{array}{cc}
\kappa_{\alpha\rho+1, \alpha\rhohat}(\lambda)+\dfrac{\sin(\pi\alpha\rhohat)}{\sin(\pi\alpha\rho)}\kappa'_{\alpha\rhohat, \alpha\rho+1}(0+) &  
-\dfrac{\sin(\pi\alpha\rhohat)}{\sin(\pi\alpha\rho)}\dfrac{\kappa_{\alpha\rhohat, \alpha\rho+1}(\lambda)}{\lambda}\\
&\\
- \dfrac{\sin(\pi\alpha\rho)}{\sin(\pi\alpha\rhohat)}\dfrac{\kappa_{\alpha\rho, \alpha\rhohat+1}(\lambda)}{\lambda}&
  \kappa_{\alpha\rhohat+1, \alpha\rho}(\lambda) 
+
 \dfrac{\sin(\pi\alpha\rho)}{\sin(\pi\alpha\rhohat)}\kappa'_{\alpha\rho, \alpha\rhohat+1}(0+)
\end{array}
\right]}
,
\end{eqnarray*}
for $\lambda\geq 0$.
\bigskip

\noindent (ii) Up to a multiplicative constant,
the dual ascending ladder MAP exponent is given by 
\begin{eqnarray*}
\hat{\boldsymbol{\kappa}}(\lambda) &=& 
\scriptsize{\left[
\begin{array}{cc}
\kappa_{\alpha\rhohat+1, \alpha\rho}(\lambda + 1-\alpha)+\dfrac{\sin(\pi\alpha\rho)}{\sin(\pi\alpha\rhohat)}\kappa'_{\alpha\rho, \alpha\rhohat+1}(0+) &  
-\dfrac{\kappa_{\alpha\rho, \alpha\rhohat+1}(\lambda+ 1-\alpha)}{\lambda+ 1-\alpha}\\
&\\
-\dfrac{\kappa_{\alpha\rhohat, \alpha\rho+1}(\lambda+ 1-\alpha)}{\lambda+ 1-\alpha}&
  \kappa_{\alpha\rho+1, \alpha\rhohat}(\lambda+ 1-\alpha) 
+
 \dfrac{\sin(\pi\alpha\rhohat)}{\sin(\pi\alpha\rho)}\kappa'_{\alpha\rhohat, \alpha\rho+1}(0+)
\end{array}
\right]},
\end{eqnarray*}
for $\lambda \geq 0$.
\end{theorem}

The next theorem deals with the case that $\alpha\in(1,2)$. For this we need to introduce another family of Bernstein functions. Define 
\[
\phi_{q+i,p+j}(\lambda) =\int_0^\infty (1-{\rm e}^{-\lambda u})\left\{\frac{((q+i)\vee (p+j)-1) }{(1-{\rm e}^{-u})^{q+i}(1+{\rm e}^{-u})^{p+j}} - \frac{(\alpha-1)}{2(1-{\rm e}^{-u})^{q}(1+{\rm e}^{-u})^{p}}\right\}{\rm e}^{- u}\d u,
\]
for $\lambda\geq 0$,  $q,p\in\{\alpha\rho, \alpha\rhohat\}$ and $i,j\in\{0,1\}$ such that $q+p = \alpha$ and  $i+j = 1$.
Note, again, that the density in curly brackets can easily be verified to be positive in all cases and is a Bernstein function since, as before,  the associated L\'evy density behaves like either $x^{-\alpha\rho -1}$ or $x^{-\alpha\rhohat-1}$ as $x\downarrow0$ and like ${\rm e}^{- x}$ as $x\uparrow\infty$. Once again, it is also subsequently straightforward to verify that the mean value $\kappa'_{q+i,p+j}(0+)$ is finite.

\begin{theorem}
\label{WHFMAPHG>1} When $\alpha\in(1,2)$,
we have the following two components to the factorisation in Theorem \ref{WHF}.

\bigskip

\noindent (i) Up to a multiplicative constant, the ascending ladder MAP exponent is given by 
\begin{eqnarray*}
\boldsymbol{\kappa}(\lambda) &=& \scriptsize{
  \left[
  \begin{array}{cc}
   \sin(\pi\alpha\rho)\phi_{\alpha\rho +1, \alpha\rhohat}(\lambda+\alpha -1)
+ \sin(\pi\alpha\rho)\phi_{\alpha\rhohat, \alpha\rho +1}'(0+)

    & 
 -\sin(\pi\alpha\rhohat)\dfrac{\phi_{\alpha\rhohat, \alpha\rho +1}(\lambda+\alpha-1)}{\lambda+\alpha-1}
    \\
    &\\
 - \sin(\pi\alpha\rho)\dfrac{\phi_{\alpha\rho, \alpha\rhohat +1}(\lambda+\alpha-1)}{\lambda+\alpha-1}

    & \sin(\pi\alpha\rhohat)\phi_{\alpha\rhohat +1, \alpha\rho}(\lambda+\alpha-1)+ \sin(\pi\alpha\rhohat)\phi_{\alpha\rho, \alpha\rhohat +1}'(0+)

  \end{array} 
  \right]},
\end{eqnarray*}
for $\lambda\geq 0$.

\bigskip

\noindent (ii) Up to a multiplicative constant, 
the dual ascending ladder MAP exponent is given by 
\begin{eqnarray*}
\hat{\boldsymbol{\kappa}}(\lambda)& = &
\scriptsize{
  \left[
  \begin{array}{cc}
   \sin(\pi\alpha\rhohat)\phi_{\alpha\rhohat +1, \alpha\rho}(\lambda)
+ \sin(\pi\alpha\rhohat)\phi_{\alpha\rho, \alpha\rhohat +1}'(0+)

    & 
 -\sin(\pi\alpha\rhohat)\dfrac{\phi_{\alpha\rho, \alpha\rhohat +1}(\lambda)}{\lambda}
    \\
    &\\
 - \sin(\pi\alpha\rho)\dfrac{\phi_{\alpha\rhohat, \alpha\rho +1}(\lambda)}{\lambda}

    & \sin(\pi\alpha\rho)\phi_{\alpha\rho +1, \alpha\rhohat}(\lambda)+ \sin(\pi\alpha\rho)\phi_{\alpha\rhohat, \alpha\rho +1}'(0+)

  \end{array} 
  \right]},
\end{eqnarray*}
for $\lambda\geq 0$.

\end{theorem}

The two main results above, and in particular the techniques used to prove them, offer many new insights into the analysis of stable processes. Classically, the Wiener--Hopf factorisation of L\'evy processes provides the basis of many proofs for fluctuation identities, both exact and asymptotic. Historically there has been less exploration in this respect for the case of MAPs. But, nonetheless, the same importance of the role of the Wiener--Hopf factorisation applies, with many proofs following analogous lines of reasoning to the L\'evy case. See for example the Appendix in \cite{P_0}. When one now takes account of the degree of explicit detail that we offer here with regard to the Lamperti-stable MAP Wiener--Hopf factorisation, as well as the pathwise embedding of the fluctuations of this MAP into the fluctuations  of stable process, one should expect to gain new results for the latter family of processes. Based on the computations derived in obtaining the matrix factorisations above, we offer some results in this respect at the end of this paper. Moreover, the robustness and applicability of the techniques we develop in proving the above two theorems also plays an important role in forthcoming work; see \cite{KRS} and \cite{Deep2}. 

\section{Non-symmetric Riesz--Bogdan--Zak transform}\label{sectBZ}

A key component in proving Theorem \ref{WHFMAPHG} will be the use of the so-called {\it Riesz--Bogdan--Zak} transform which we now outline. 
\begin{theorem}[Riesz--Bogdan--Zak transform]
\label{th:BZ} Suppose that $X$ is a stable process as outlined in the introduction.
Define
\[
\eta(t) = \inf\{s>0 : \int_0^s |X_u|^{-2\alpha}{\rm d}u >t\}, \qquad t\geq 0.
\]
Then, for all $x\in\mathbb{R}\backslash\{0\}$, $(-1/{X}_{\eta(t)})_{t\geq 0}$ under $\mathbb{P}_{x}$ is equal in law to $(X, \mathbb{P}_{-1/x}^\circ)$, where  
\begin{equation}
\left.\frac{{\rm d}\mathbb{P}^\circ_x}{{\rm d}\mathbb{P}_x}\right|_{\mathcal{F}_t} = \left(\frac{\sin(\pi\alpha\rho) + \sin(\pi\alpha\hat\rho)-(\sin(\pi\alpha\rho) - \sin(\pi\alpha\hat\rho) ){\rm sgn}(X_t)}{
\sin(\pi\alpha\rho) + \sin(\pi\alpha\hat\rho)-(\sin(\pi\alpha\rho) - \sin(\pi\alpha\hat\rho)){\rm sgn}(x)}\right)\left|\frac{X_t}{x}\right|^{\alpha -1}\mathbf{1}_{(t<\tau^{\{0\}})}
\label{updownCOM}
\end{equation}
and $\mathcal{F}_t := \sigma(X_s: s\leq t)$, $t\geq 0$. Moreover, the process $(X, \mathbb{P}^\circ_x)$, $x\in\mathbb{R}\backslash\{0\}$ is a self-similar Markov process with underlying MAP via the Lamperti-Kiu transform given by 
\begin{equation}
\boldsymbol{F}^\circ(z) =
 \left[
  \begin{array}{cc}
    - \dfrac{\Gamma(1-z)\Gamma(\alpha+z)}
      {\Gamma(1-\alpha\rho-z)\Gamma(\alpha\rho+ z)}
    & \dfrac{\Gamma(1-z)\Gamma(\alpha+z)}
      {\Gamma(\alpha\rho)\Gamma(1-\alpha\rho)}
    \\
    &\\
    \dfrac{\Gamma(1-z)\Gamma(\alpha+ z)}
      {\Gamma(\alpha\hat\rho)\Gamma(1-\alpha\hat\rho)}
    & - \dfrac{\Gamma(1-z)\Gamma(\alpha+z)}
      {\Gamma(1-\alpha\hat\rho-z)\Gamma(\alpha\hat\rho+z)}
  \end{array} 
  \right],
  \label{Fcirc}
\end{equation}
for $\Re(z)\in(-\alpha,1).$
\end{theorem}

In the case that $X$ is a symmetric stable process (i.e. $c_+ = c_-$, equivalently $\rho=1/2$) the result is contained in the result of Bogdan and Zak \cite{BZ}, who deal  with isotropic stable processes in one or more dimensions. One may see the work of Bogdan and Zak, specifically the idea of spatial inversion through a sphere (here an interval),  as building on original results of M. Reisz, who used this technique to analyse potentials, cf. \cite[pp. 13-171]{R1}, \cite{R2} as well as the discussion in Section 3 of \cite{BGR}. 

It is straightforward  to deduce that $(X, \mathbb{P}^\circ_x)$, $x\in\mathbb{R}\backslash\{0\}$ is a rssMp, inheriting the index of self-similarltiy $\alpha$ from $(X, \mathbb{P}_x)$, $x\in \mathbb{R}\backslash\{0\}$. When $\alpha\in(1,2)$, \cite{CPR} have identified $(X, \mathbb{P}^\circ_x)$ to be the law of a stable process conditioned to avoid the origin when issued from $x\in\mathbb{R}\backslash\{0\}$. Their requirement that $\alpha\in(1,2)$ pertains to the fact that points are polar for $\alpha\in(0,1]$ and, accordingly, conditioning to avoid the origin makes no sense in the latter parameter regime. Nonetheless, the change of measure (\ref{updownCOM}) is still meaningful and gives preference to paths that approach the origin closely,  penalising paths that wander far from the origin. In fact, we shall see in due course from its Lamperti--Kiu representation that  $(X, \mathbb{P}^\circ_x)$, $x\in\mathbb{R}\backslash\{0\}$, is absorbed at the origin almost surely; see the forthcoming Remark \ref{absorbedat0}. In this sense, $(X, \mathbb{P}^\circ_x)$, $x\in\mathbb{R}\backslash\{0\}$, may be considered to be the stable process conditioned to be absorbed at the origin.  When $\alpha = 1$, one easily sees that $ \mathbb{P}^\circ_x =  \mathbb{P}_x$, $x\in\mathbb{R}$.

In order to prove Theorem \ref{th:BZ} we first need to briefly discuss the analogue of the exponential change of measure and  Esscher transform for MAPs.
Referring back to (\ref{MAPHG}), for each $z\in\mathbb{C}$ such that $\Re(z)\in(-1,\alpha)$, there exists a  leading real-valued eigenvalue of the matrix $\boldsymbol{F}(z)$, also called the
{\it Perron--Frobenius eigenvalue};
see \cite[\S XI.2c]{Asm-apq2} and \cite[Proposition 2.12]{Iva-thesis}.
If we denote this eigen value by $\chi(z)$, then it turns out that it
 is larger than the real part of all its other eigenvalue.
Furthermore, the
corresponding
right-eigenvector $\boldsymbol{v}(z)$ has strictly positive entries,
and can be normalised such that
$  \boldsymbol\pi \cdot \boldsymbol{v}(z) = 1$,
where we recall that $\boldsymbol\pi$ is the stationary distribution of
the underlying chain $J$.

The leading eigenvalue $\chi(z)$ features in the following probabilistic
result, which identifies a martingale (the analogue of the Wald martingale), a change of measure and the analogue of the Esscher transformation
for exponents of  L\'evy processes; cf.\ \cite[Proposition XI.2.4, Theorem XIII.8.1]{Asm-apq2}.
(Note that the result is still true  for general MAPs as introduced in Section \ref{rssmp}.)
\begin{proposition}
\label{p:mg and com}
Let $\mathcal{G}_{t} = \sigma\{(\xi(s), J(s)): s\leq t\}$, $t\geq 0$, and
\begin{equation} M(t,\gamma) =  {\rm e}^{\gamma (\xi(t)-\xi(0)) - \chi(\gamma)t}
    \frac{v_{J(t)}(\gamma)}{v_{J(0)}(\gamma)} ,
  \for t \ge 0, 
  \label{MAPCOM}\end{equation}
for some $\gamma$ such that $\chi(\gamma)$  is defined.
Then,
  $M(\cdot,\gamma)$ is a unit-mean martingale with respect to $\GGt$. Moreover, under the change of measure 
  \[
  \left.\frac{{\rm d}\mathbb{P}^\gamma_{x,i}}{{\rm d}\mathbb{P}_{x,i}}\right|_{\mathcal{G}_t} = M(t,\gamma),\qquad t\geq 0,
  \]
  the process $(\xi,J)$ remains in the class of MAPs and, where defined, its  characteristic exponent given by 
  \begin{equation}
  \boldsymbol{F}_\gamma(z) = \boldsymbol{\Delta}_{\boldsymbol v}(\gamma)^{-1}\boldsymbol{F}(z+\gamma)\boldsymbol{\Delta}_{\boldsymbol v}(\gamma) - \chi(\gamma)\mathbf{I}, 
  \label{Esscher}
  \end{equation}
  where $\mathbf{I}$ is the identity matrix and $\boldsymbol{\Delta}_{\boldsymbol v}(\gamma) = {\rm diag}(\boldsymbol{v}(\gamma))$. (The latter matrix we understand to mean the diagonal matrix with entries of $\boldsymbol{v}(\gamma)$ loaded on to its diagonal.)
  \end{proposition}

\begin{proof}[Proof of Theorem \ref{th:BZ}]
When we take $\boldsymbol F$ to be given by  (\ref{MAPHG}), we can compute explicitly the quantity $\boldsymbol\pi$ as well a $\boldsymbol{v}(\gamma)$ for a particular value of $\gamma$ that is of interest. We are interested in the case that $
gamma: = \alpha-1$. Note that $\gamma \in (-1,\alpha)$. A straightforward computation shows that, for $\Re(z)\in (-1,\alpha)$,
\[
{\rm det}\boldsymbol{F}(z)  =\frac{\Gamma(\alpha-z)^2\Gamma(1+z)^2}{\pi^2}\left\{
  \sin(\pi(\alpha\rho- z))\sin(\pi(\alpha\hat\rho- z)) 
  - \sin(\pi \alpha \rho) \sin(\pi\alpha\hat\rho)\right\},
\]
which has a root at $z = \alpha-1$. In turn, this implies that $\chi(\alpha-1)=0$. One also easily checks with the help of the reflection formula for gamma functions that 
\[
\boldsymbol{v}(\alpha -1) \propto\left[
\begin{array}{c}
\sin(\pi\alpha\hat\rho)\\
\sin(\pi\alpha\rho)
\end{array} \right]
\]
and, by considering $\boldsymbol{F}(0) = \boldsymbol{Q}$,
\begin{equation}
\boldsymbol{\pi} \propto
\left[
\begin{array}{c}
\sin(\pi\alpha\rho)\\
\sin(\pi\alpha\hat\rho) 
\end{array}
\right].
\label{pi}
\end{equation}
We see that with $\gamma = \alpha -1$, the change of measure (\ref{MAPCOM}) corresponds precisely to (\ref{updownCOM}) when $(\xi, J)$ is the MAP underlying the stable process. In particular, we can now say that the MAP associated to the process $(X, \mathbb{P}^\circ_x)$, $x\in\mathbb{R}\backslash\{0\}$, formally named $\boldsymbol{F}^\circ(z)$,  is equal to
\[
\boldsymbol{F}_{\alpha-1}(z) = \left[
  \begin{array}{cc}
    - \dfrac{\Gamma(1-z)\Gamma(\alpha+z)}
      {\Gamma(1-\alpha\rho-z)\Gamma(\alpha\rho+ z)}
    & \dfrac{\Gamma(1-z)\Gamma(\alpha+z)}
      {\Gamma(\alpha\rho)\Gamma(1-\alpha\rho)}
    \\
    &\\
    \dfrac{\Gamma(1-z)\Gamma(\alpha+ z)}
      {\Gamma(\alpha\hat\rho)\Gamma(1-\alpha\hat\rho)}
    & - \dfrac{\Gamma(1-z)\Gamma(\alpha+z)}
      {\Gamma(1-\alpha\hat\rho-z)\Gamma(\alpha\hat\rho+z)}
  \end{array} 
  \right],
\]
for $\Re(z)\in(-\alpha,1)$, where we have again used the reflection formula for the the gamma function to deal with the terms coming from $\boldsymbol\Delta_{\boldsymbol\upsilon}(\alpha-1)$ in (\ref{Esscher}).

Now let us turn our attention to the process $(-1/X_{\eta(t)})_{t\geq 0}$. First note that, if $X$ is an $(\alpha,\rho)$ stable process, then  $-X$ is a $(\alpha,\hat\rho)$ stable process.  Next, we show that $(- 1/X_{\eta(t)})_{t\geq 0}$ is a rssMp with index $\alpha$ by analysing its Lampert--Kiu decomposition. 

To this end, note that, if $(\xi^*, {J}^*)$ is the MAP that underlies  ${X}^*: =-X$, then its matrix exponent, say ${\boldsymbol F}^*(z)$, is equal to (\ref{MAPHG}) with the roles of $\rho$ and $\hat\rho$ interchanged. As ${X}^*$ is a rssMp, we have 
\[
{X}^*_t = \exp\left\{
{\xi}^*({\varphi}^*(t)) + \iu \pi (J^*(\varphi^*(t)) +1)
\right\},\qquad t<\tau^{\{0\}},
\]
where 
\[
\int_0^{{\varphi}^*(t)}{\rm e}^{\alpha{\xi}^*(s)}{\rm d}s = t. 
\]
Noting that 
\[
\int_0^{\eta(t)} {\rm e}^{-2\alpha {\xi}^*(\varphi^*(u))}
{\rm d}u = t, \qquad \eta(t)<\tau^{\{0\}},
\]
a straightforward differentiation of the last two integrals shows  that, respectively, 
\[
\frac{{\rm d}\varphi^*(t)}{{\rm d}t} = {\rm e}^{-\alpha {\xi}^*({\varphi}^*(t))}\text{ and }
\frac{{\rm d}\eta(t)}{{\rm d}t} ={\rm e}^{2\alpha {\xi}^*({\varphi}^*\circ\eta(t))}, \qquad\eta( t)<\tau^{\{0\}}.
\]
The chain rule now tells us that 
\[
\frac{{\rm d}({\varphi}^*\circ\eta)(t)}{{\rm d}t} = \left.\frac{{\rm d}\varphi^*(s)}{{\rm d}s}\right|_{s = \eta(t)}\frac{{\rm d}\eta(t)}{{\rm d}t}  = {\rm e}^{\alpha {\xi}^*({\varphi}^*\circ\eta(t))},
\]
and hence,
\[
\int_0^{{\varphi}^*\circ\eta(t)}{\rm e}^{-\alpha {\xi}^*(u)} {\rm d}u = t, \qquad \eta(t)<\tau^{\{0\}}.
\]
The qualification that $\eta(t)<\tau^{\{0\}}$ only matters when $\alpha\in(1,2)$. In that case, the fact that $\mathbb{P}_x(\tau^{\{0\}}<\infty) = 1$ for all $x\in\mathbb{R}$ implies that  $\lim_{t\to\infty}\xi^*_t = -\infty$ almost surely. As a consequence, it follows that $\int_0^\infty{\rm e}^{-\alpha {\xi}^*(u)} {\rm d}u = \infty $ and hence $\lim_{t\to\infty}{\varphi}^*\circ\eta(t) = \infty$. That is to say, we have $\lim_{t\to\infty}\eta(t)= \tau^{\{0\}}$.
Noting that for $k\in\mathbb{N}$, ${\rm e}^{-\iu\pi k} ={\rm e}^{\iu\pi k}$, it now follows that 
\[
\frac{1}{{X}^*_{\eta (t)}} 
=  \exp\left\{ -{\xi}^*({\varphi}^*\circ\eta(t)) +\iu\pi (J^*({\varphi}^*\circ\eta (t))+1)\right\} , \qquad t<\tau^{\{0\}}
\]
is the representation of a rssMp whose underlying MAP has matrix exponent given by ${\boldsymbol{F}}^*(-z)$, whenever it is well defined. Recalling the definition of ${\boldsymbol F}^*(z)$, we see that the MAP that underlies $(-1/X_{\eta(t)})_{t\geq 0}$ via the Lamperti--Kiu transform is identically equal in law to the MAP with matrix exponent $\boldsymbol{F}_{\alpha-1}(z)$.
The proof is now complete.
\end{proof}

\section{The ascending ladder MAP}\label{ascending}

We shall derive the Matrix exponent $\boldsymbol{\kappa}$ by deriving each and every component of the matrices ${\rm diag}(\Phi_1(\lambda), \Phi_2(\lambda))$, ${\boldsymbol{\Lambda}}$ and $\boldsymbol{K}(\lambda)$, for $\lambda \geq 0$. In order to do this, we will make use of the Riesz--Bogdan--Zak transform from the previous section as well as some classical Markov additive renewal theory. 
In order to understand how the latter bears relevance, we need to briefly recall how the ascending ladder process $(H, J^+)$ emerges as a consequence of excursion theory and accordingly 
 is a non-decreasing MAP.

 Let  $Y^{(x)}_t = (x\vee\bar\xi(t)) - \xi(t)$, $t\geq 0$, where  $\bar\xi(t) = \sup_{s\leq t}\xi(s)$, $t\geq 0$. 
Following ideas that are well known from the theory of L\'evy processes, it is straightforward to show that, as a pair, the process $(Y^{(x)},J)$ is a strong Markov process.
For convenience, write $Y$ in place of $Y^{(0)}$.
We know by standard theory (c.f. Chapter   IV of \cite{BertoinLP}) there exists a local  time  of $(Y,J)$ at the point $(0,i)$, which we henceforth denote by $\{{L}^{(i)}_t: t\geq 0\}$. Now consider the process
\[
{L}_t := \sum_{i\in E} {L}^{(i)}_t, \qquad t\geq 0.
\]
Note that this local time can be constructed uniquely up to a multiplicative constant, which will turn out to be of pertinence later.
Since, almost surely, for each $i\neq j$ in $E$, the points of increase of  ${L}^{(i)}$ and ${L}^{(j)}$ are disjoint, it follows that  $({L}^{-1}, {H}^+, J^+): = \{({L}^{-1}_t, {H}^+(t), J^+(t)):t\geq 0\}$ is a (possibly killed)   {\it Markov additive bivariate subordinator}, where	
\[
{H}^+(t) : = \xi({L}^{-1}_t)\text{ and } J^+(t) : = J({L}^{-1}_t), \qquad \text{ if } {L}^{-1}_t<\infty,
\]
and ${H}^+(t) : = \infty$ and $J^+(t) : = \dagger$ (a cemetery state) otherwise.
Note, as a Markov additive subordinator, $({L}^{-1}, {H}^+, J^+)$ is a Markov additive process with co-ordinatewise non-decreasing paths such that, in each of the states of $J^+$,  the process $H$ evolves as a subordinator possibly killed at an independent and exponentially distributed time, whereupon it is sent to the cemetery state $\infty$ and $J^+$ is sent to $\dagger$. Killing rates may depend on the state of $J^+$. 

If we define
\[
\epsilon_t  = \{\epsilon_t(s): = \xi({L}^{-1}_{t-} + s) - \xi({L}^{-1}_{t-}) : s\leq {\Delta} {L}^{-1}_t\}, \qquad \text{ if }{\Delta} {L}^{-1}_{t}>0,
\]
and $\epsilon_t = \partial$, some artificial isolated state, otherwise, then it turns out that the process $\{\epsilon_t: t\geq 0\}$ is a (killed) Cox process.
Henceforth, write $n_i$ for the intensity measure of this Cox process when the underlying modulating chain $J^+$ is in state $i\in E$.
As a Markov additive subordinator, the process $({H}^+, J^+)$ has a matrix exponent given by
\[
{\mathbf E}_{0,i}\big[{\rm e}^{- \lambda {H}^+(t) }, J^+(t) = j\big] = \big({\rm e}^{- \boldsymbol{\kappa}(\lambda)t}\big)_{i,j},\qquad \lambda\geq 0,
\]
where  $\kappa^+(\lambda)$ was given in (\ref{MAPBernstein}).
Note in particular that, for $i=1,2$,  $\Phi_i( \lambda)$ is the subordinator Bernstein exponent that describes the movement of  ${H}^+$ when the modulating chain $J^+$ is in state $i$. Moreover, ${\boldsymbol{\Lambda}}$ is the intensity of $J^+$ and the matrix ${\boldsymbol{K}}(\lambda) = ({\boldsymbol{K}}(\lambda))_{i,j}$ is such that, for $i\neq j$ in $E$, its $(i,j)$-th entry is the Laplace transform of the additional jump incurred by  $H$ when the modulating chain changes state from $i$ to $j$. The diagonal elements of $\boldsymbol{K}(\lambda)$ are set to unity. 
In general, we can write 
\[
\Phi_i(\lambda) = {n}_i(\zeta = \infty) +  \texttt{b}_i \lambda +\int_0^\infty (1-{\rm e}^{-\lambda x})n_i( \epsilon_\zeta \in {\rm d}x, J(\zeta) = i, \zeta<\infty),\qquad \lambda\geq 0, 
\]
where $\texttt{b}_i\geq 0$ and $\zeta = \inf\{s\geq 0: \epsilon(s) >0\}$ for the canonical excursion $\epsilon$. 
For the case of the Lamperti-stable MAP, on account of the fact that the stable processes we consider in this paper do not creep, we can immediately set $ \texttt{b}_i = 0$ for $i=1,2$. In the case that $\alpha\in(0,1]$, points are polar and  the underlying stable process explores arbitrarily large distances from the origin. When $\alpha\in(1,2)$, the MAP in question represents the stable process until absorption at the origin, which occurs almost surely. We cannot rely on the range of the stable process until this time being unbounded and therefore   ${n}_i(\zeta = \infty) >0$, for $i=1,2$. 

Let  us recall the following result, which is a special case of a general Markov additive renewal limit theorem given in the Appendix of \cite{P_0}. Such limit theorems are classical and can be found in many other contexts; see \cite{lalley}, \cite{kesten} and \cite{Alsmeyer, Alsmeyer1} to name but  a few.
First we fix some notation. For each $a>0$, let 
\[
T_a = \inf\{t>0 : H^+(t) >a\}.
\]

\begin{lemma}\label{MRT} Suppose that the ladder height process $(H^+, J^+)$ does not experience killing, that is to say, $L_\infty =\infty$.
Then for $x>0$ and $i,j \in \{1,2\}$, 
\begin{eqnarray}
\lefteqn{\lim_{a\to\infty} \mathbf{P}_{0,i}(H^+({T_a})-a \in {\rm d}x, J^+({T_a} )= j) }&&\notag\\
&&= \frac{1}{{\mathbf E}_{0,\pi}({H}^+(1))}\Big[ \pi_jn_j(\epsilon({\zeta}) > x, J(\zeta) =j, \zeta<\infty) + \pi_k {\Lambda}_{k,j} (1-F^+_{k,j}(x))\Big]{\rm d}x,
\label{decompose}
\end{eqnarray}
where  $k\in\{1,2\}$ is such that $k\neq j$  and $\int_{[0,\infty )}{\rm e}^{-\lambda x}F^+_{k,j}({\rm d}x)=\mathbf{E}[{\rm e}^{-\lambda U^+_{k,j}}]$.
As a more refined version of the above statement, we also have that 
\begin{eqnarray}
\lefteqn{\lim_{a\to\infty} \mathbf{P}_{0,i}(H^+({T_a})-a \in {\rm d}x, J^+({T_a} )= j, J^+({T_a-} )= j) }&&\notag\\
&&= \frac{1}{{\mathbf E}_{0,\pi}({H}^+(1))} \pi_jn_j(\epsilon({\zeta}) > x, J(\zeta) =j, \zeta<\infty)
{\rm d}x,
\label{decompose_only_n}
\end{eqnarray}
For all limits above, we interpret the right hand side as zero when $\mathbf{E}_{0,\pi}({H}^+(1))  = \infty$.
\end{lemma}

Recall that the local time $L$ can be constructed up to an arbitrary multiplicative constant. If one follows how this constant permeates through to the definition of $n_i$, $i=1,2$ and $\Lambda_{i,j}$, $i,j=1,2$, we see that they are also defined up to the same multiplicative constant. For this reason, we shall, without loss of generality assume that this constant is chosen such that $\mathbf{E}_{0,\pi}({H}^+(1))  =1$.

As will be explained in the forthcoming computations, the above Lemma provides the key to picking out the individual components that contribute to the matrices $\boldsymbol{\kappa}(\lambda)$ and $\boldsymbol{\hat{\kappa}}(\lambda)$ by decomposing the left-hand side of (\ref{decompose}) and (\ref{decompose_only_n}) in terms of the behaviour of the associated stable process. Ultimately what we shall see is that all computations boil down to identities that come from the so-called two-sided exit problem for the stable process, which is originally due to \cite{Rog} (see also Exercise 7.7 of \cite{Kyp}).

To elaborate in a little more detail, let us introduce the stopping times for the stable process: for each $a\in\mathbb{R}$, 
\[
\tau^+_a = \inf\{t>0: X_t >a\} \text{  and } \tau^-_a = \inf\{t>0: X_t <a\}.
\]

\begin{theorem}\label{2sidedexit}
Suppose that $X$ is a stable process as described in the introduction. Then, 
for $\theta\geq 0$ and $x\in(0,1)$,
\begin{eqnarray*}
\lefteqn{\mathbb{P}_x(X_{\tau^+_1}-1 \in\d  \theta ;
\tau^+_1<\tau^-_0)}&&\\
&&= \frac{\sin(\pi\alpha \rho)}{\pi} (1-
x)^{\alpha\rho}x^{\alpha\rhohat}
\theta^{-\alpha\rho}(\theta+1)^{-\alpha\rhohat}(\theta+1-x)^{-1}\d\theta.
\end{eqnarray*}
Equivalently, by scaling and translation, for $\theta\geq 0$ and $x\in (-1,1)$, 
\begin{eqnarray*}
\lefteqn{\mathbb{P}_x(X_{\tau^+_1}-1 \in\d  \theta ;
\tau^+_1<\tau^-_{-1})}&&\\
&&= \frac{\sin(\pi\alpha \rho)}{\pi} (1-
x)^{\alpha\rho}(1+x)^{\alpha\rhohat}
\theta^{-\alpha\rho}(\theta+2)^{-\alpha\rhohat}(\theta+1-x)^{-1}{\d}\theta.
\end{eqnarray*}
\end{theorem}
\noindent Note, we have given two forms of the expression in the theorem   above purely for convenience as they will both be used in the computations below.

\section{The ascending MAP with $\alpha\in(0,1]$}\label{smalla}

When $\alpha\in(0,1]$, we have that points are polar for the stable process $X$. In particular, as a rssMp,  the origin is not accessible. Recalling that MAPs respect the same trichotomy as L\'evy processes in terms of drifting and oscillating, we thus have that the Lamperti-stable  MAP satisfies $\limsup_{t\to\infty}\xi(t)=\infty$.  This means that the conditions of Lemma \ref{MRT} are satisfied. 

In order to apply the aforesaid lemma, write $X^{(x)}$ to indicate the initial value of the stable process, i.e. $X_0^{(x)} =x\in\mathbb{R}\backslash\{0\}$. Thanks to self-similarity, and the Lamperti--Kiu representation, we have, for example, that, on the event $\{J^+({T_a}) = 1\}$, i.e. $\{X_{\tau^+_{{\rm e}^a} \wedge \tau^-_{-{\rm e}^a}} >{\rm e}^a \}$,
\begin{equation}
\exp\{H^+({T_a})-a \} = \frac{X^{(x)}_{\tau^+_{{\rm e}^a} \wedge \tau^-_{-{\rm e}^a}}}{{\rm e}^a}=^d X^{(x{\rm e}^{-a})}_{\tau^+_{1} \wedge \tau^-_{-1}}
\label{scaleit}
 \end{equation}
 Hence, taking limits,  we have, for $i, j\in\{1,2\}$ and $u>0$,
 \[
 \lim_{a\uparrow \infty}\mathbf{P}_{0,i}(H^{+}(T_{a})-a >u , J^{+}(T_{a})=1)=\lim_{x\to0}\mathbb{P}_x(X_{\tau^+_1}>{\rm e}^{u}, \tau_{1}^{+}<\tau_{-1}^{-}) .
 \]
Moreover, if we write $\overline{X}_t = \sup_{s\leq t}X_s$ and $\underline{X}_t = \inf_{s\leq t}X_s$, $t\geq 0$, then we also have
\begin{eqnarray}
\lefteqn{\pi_1n_1(\epsilon({\zeta}) > u, J(\zeta) =1, \zeta<\infty)}&&\notag\\
&& =-\frac{\d }{\d u} \lim_{x\to0}\mathbb{P}_x\left(X_{\tau^+_1}>{\rm e}^u, \overline{X}_{\tau^+_1- }>|\underline{X}_{\tau^+_1-}|, \tau^+_1<\tau^-_{-1}\right)\notag\\
&&=-\frac{\d}{\d u} \int_0^1 \mathbb{P}(X_{\tau^+_1}>{\rm e}^u, \overline{X}_{\tau^+_1- }\in \d z,  \tau^+_1<\tau^-_{-z} ).
\label{dissleadto}
\end{eqnarray}

In order to progress our computations further and reach the goal of producing an identity for $\boldsymbol{\kappa}(\lambda)$, 
we shall first establish some intermediary results, starting with the following.
\begin{lemma}\label{PhiLemma} For $\lambda\geq 0$,
\[
\Phi_1(\lambda) =
\frac{\sin(\pi\alpha\rho) + \sin(\pi\alpha\rhohat)}{\pi}\kappa_{\alpha\rho+1,\alpha\rhohat}(\lambda).
\]
\end{lemma}

\begin{proof}
We start by  noting that, for $y\in[0,1]$, 
\begin{eqnarray*}
\lefteqn{\mathbb{P}(X_{\tau^+_1}>{\rm e}^u, \overline{X}_{\tau^+_1- } \leq y,  \tau^+_1<\tau^-_{-z})}&&\\
&&=\mathbb{P}(X_{\tau^+_y}>{\rm e}^u,\tau^+_y<\tau^-_{-z} )\\
&&=\mathbb{P}_{z/(z+y)}\left(X_{\tau^+_{1}}-1>\frac{{\rm e}^u-y}{z+y},\tau^+_{1}<\tau^-_{0} \right)\\
&&=\frac{\sin(\pi\alpha\rho)}{\pi}\int_{\frac{{\rm e}^u-y}{z+y}}^\infty 
\left(\frac{y}{y+z}\right)^{\alpha\rho}\left(\frac{z}{z+y}\right)^{\alpha\rhohat}
 t^{-\alpha\rho}(t+1)^{-\alpha\rhohat}\left( t+1 - \frac{z}{z+y}\right)^{-1}{\rm d}t,
\end{eqnarray*}
where the penultimate probability follows from scaling and the final equality follows from Theorem \ref{2sidedexit}. Hence, it follows that 
\begin{eqnarray}
\lefteqn{\pi_1n_1(\epsilon({\zeta}) > u, J(\zeta) =1, \zeta<\infty)}&&\notag\\
&&=- \int_0^1\frac{\d}{\d y}\frac{\d}{\d u} \left.\mathbb{P}(X_{\tau^+_1}>{\rm e}^u, \overline{X}_{\tau^+_1- }\leq y,  \tau^+_1<\tau^-_{-z} )\right|_{y=z}\d z\notag\\
&&=\frac{\sin(\pi\alpha\rho)}{\pi}\int_0^1{\rm e}^u \left.\frac{\d}{\d y} y^{\alpha\rho}z^{\alpha\rhohat}({\rm e}^{u}-y)^{-\alpha\rho}({\rm e}^u +z)^{-\alpha\rhohat} \right|_{y=z}\d z\notag\\
&&=\alpha\rho\frac{\sin(\pi\alpha\rho)}{\pi}\int_0^1 {\rm e}^u z^{\alpha -1}({\rm e}^u - z)^{-\alpha\rho -1}({\rm e}^u+z)^{-\alpha\rhohat}\d z
\label{similarlater}
\end{eqnarray}

Finally we can now take Laplace transforms and compute
\begin{eqnarray}
\pi_1\Phi_1(\lambda) &=&\lambda \int_0^\infty {\rm e}^{-\lambda u}
\pi_1n_1(\epsilon({\zeta})>u,J(\zeta)=1, \zeta<\infty)
\notag\\
&=&{\lambda \alpha\rho}\frac{\sin(\pi\alpha\rho)}{\pi}\int_0^1\int_0^\infty {\rm e}^{-(\lambda+\alpha) u} z^{\alpha-1} (1+z{\rm e}^{-u})^{-\alpha\hat\rho}(1-z{\rm e}^{-u})^{-(\alpha\rho+1)}\d u \, \d z\notag\\
&=&{\alpha\rho}\frac{\sin(\pi\alpha\rho)}{\pi} \int_0^\infty (1-{\rm e}^{-\lambda w})\frac{{\rm e}^{-\alpha w}}{(1+{\rm e}^{-w})^{\alpha\hat\rho}(1-{\rm e}^{-w})^{\alpha\rho+1}}\d w\notag\\
&=&
\frac{\sin(\pi\alpha\rho)}{\pi}\kappa_{\alpha\rho+1, \alpha\rhohat}(\lambda),
\label{inthespirit}
\end{eqnarray}
where,  in order to get  the third equality,  we have first  substituted $\theta = z{\rm e}^{-u}$ in the second equality, then used Fubini's Theorem and finally substituted $\theta = {\rm e}^{-w}$. The result follows once we recall that $\pi_1 = \sin(\pi\alpha\rho)/ (\sin(\pi\alpha\rho) + \sin(\pi\alpha\rhohat))$.
\end{proof}

We are now in a position to identify $\boldsymbol{\kappa}(\lambda)$ as presented in Theorem \ref{WHFMAPHG} when $\alpha\in(0,1]$.
\begin{proof}[Proof of Theorem \ref{WHFMAPHG} (i)]

In the spirit of Lemma \ref{MRT}, we also note that, again with the help of Theorem \ref{2sidedexit}, we have, for $u>0$, 
\begin{eqnarray}
\lefteqn{
\lim_{a\to\infty }\mathbf{P}_{0,i}(H^{+}(T_{a})-a\leq u  ;J^{+}(T_{a})=1)}\notag\\
&&=\mathbb{P}(X_{\tau^+_1} \leq {\rm e}^u; \tau^+_1<\tau^-_{-1})\notag\\
&&=\mathbb{P}_\frac{1}{2}(X_{\tau^+_1}-1 \leq \frac{1}{2}({\rm e}^u-1); \tau^+_1<\tau^-_{0})\notag\\
&&=\frac{\sin(\pi\alpha\rho)}{\pi} \left(\frac{1}{2}\right)^\alpha\int_0^{\frac{1}{2}({\rm e}^u-1)} t^{-\alpha\rho} (1+t)^{-\alpha\hat\rho} (t+{1}/{2})^{-1}\d t,\label{lastsection}
\end{eqnarray}
where we understand $\mathbf{P}_{0,i}$ be the law of the MAP with matrix exponent $\boldsymbol{F}$ issued from $(0,i)$.
Moreover, it follows that, for $\lambda\geq 0$,
\[
\Theta_1(\lambda) : =\lim_{a\to\infty }\mathbf{P}_{0,i}({\rm e}^{-\lambda(H^{+}(T_{a})-a)} ;J^{+}(T_{a})=1)\\
=\frac{\sin(\pi\alpha\rho)}{\pi}\int_0^\infty {\rm e}^{-\lambda u}\frac{{\rm e}^{-\alpha u}}{ (1-{\rm e}^{-u})^{\alpha\rho}(1+{\rm e}^{-u})^{\alpha\hat\rho} }\d u.
\]

Suppose now we define 
$
f(x) = {\rm e}^{-\alpha x} (1-{\rm e}^{-x})^{-\alpha\rho}(1+{\rm e}^{-x})^{-\alpha\hat\rho} .
$
A straightforward computation shows that 
\[
f'(x)
 =- f(x)\left\{ \frac{\alpha\rho}{(1-{\rm e}^{-x})} + \frac{\alpha\hat\rho}{(1+{\rm e}^{-x})}\right\}.
\]
This will  useful in the following computation, which also uses the conclusion of Lemmas \ref{MRT} and  \ref{PhiLemma} as well as   integration by parts:
\begin{eqnarray*}
\pi_2\Lambda_{2,1}\boldsymbol{K}(\lambda)_{2,1}
&=&\Theta_1(\lambda)- {\pi_1}\frac{\Phi_1(\lambda)}{\lambda}\\
&=&\frac{\sin(\pi\alpha\rho)}{\pi}\left\{
 - \int_0^\infty \frac{(1-{\rm e}^{-\lambda x})}{\lambda}f'(x)\d x
-\int_0^\infty \frac{(1-{\rm e}^{-\lambda x})}{\lambda}f(x)\frac{\alpha\rho }{(1-{\rm e}^{-x})}\d x\right\}\\
&=&
\frac{\sin(\pi\alpha\rho)}{\pi}\frac{\kappa_{\alpha\rho, \alpha\rhohat+1}(\lambda)}{\lambda},
\end{eqnarray*}
for $\lambda\geq 0$.

Summarising  the above computations, as well as the statement of Lemma \ref{PhiLemma}, we have that, up to the multiplicative constant 
$(\sin(\pi\alpha\rho)+\sin(\pi\alpha\rhohat))/\pi$,
\[
\Phi_1(\lambda) =
\kappa_{\alpha\rho+1, \alpha\rhohat}(\lambda)\text{ and }\Lambda_{2,1}\boldsymbol{K}(\lambda)_{2,1} = 
\frac{\sin(\pi\alpha\rho)}{\sin(\pi\alpha\rhohat)}\frac{\kappa_{\alpha\rho, \alpha\rhohat+1}(\lambda)}{\lambda}, \qquad \lambda\geq 0,
\]
and hence, again up to the same multiplicative constant, 
\[
\Lambda_{2,1} = 
\frac{\sin(\pi\alpha\rho)}{\sin(\pi\alpha\rhohat)}\kappa'_{\alpha\rho, \alpha\rhohat + 1}(0+).
\]
By exchanging the roles of $\rho$ and $\rhohat$, we similarly get expressions for $\Phi_2(\lambda)$, $\Lambda_{1,2}\boldsymbol{K}(\lambda)_{1,2}$  and $\Lambda_{1,2}$.
Putting the pieces together into (\ref{MAPBernstein}) we have the required form for $\boldsymbol{\kappa}(\lambda)$.
\end{proof}

\section{The ascending MAP with $\alpha\in(1,2)$}\label{ka>1}

When $\alpha\in(1,2)$ the computations in the previous section break down as we must take account of the fact that the Lamperti--Kiu transform now only describes the stable process up to the first hitting of the origin.  The way we will deal with this is to take advantage of a trick that emerges from the Esscher transform for MAPs. 

Revisiting Proposition \ref{p:mg and com} and the proof of the Riesz--Bogdan--Zak transform in Theorem \ref{th:BZ}, let us consider the MAP corresponding to $(X,\mathbb{P}^\circ_x)$, $x\in\mathbb{R}\backslash\{0\}$. Recall that its matrix exponent, $\boldsymbol{F}^\circ(z)$,   was given by \eqref{Fcirc}. As well as being expressed via an Esscher transform of $\boldsymbol{F}(z)$ (see the proof of Theorem \ref{th:BZ}), it also enjoys a Wiener--Hopf factorisation so that 
\begin{eqnarray}
 - \boldsymbol{F}^\circ(\iu \theta) &=& \boldsymbol{\Delta}_{\boldsymbol{\pi}^\circ}^{-1}\hat{\boldsymbol{\kappa}}^\circ( \iu \theta)^{\rm T}\boldsymbol{\Delta}_{\boldsymbol{\pi}^\circ}\boldsymbol{\kappa}^\circ(-{\rm i}\theta) \notag\\
 &=& - \boldsymbol{\Delta}_{\boldsymbol{\upsilon}}(\alpha-1)^{-1} \boldsymbol{F}({\rm i}\theta +\alpha-1) \boldsymbol{\Delta}_{\boldsymbol{\upsilon}}(\alpha -1), 
 \label{2WHFs}
\end{eqnarray}
for $z\in\mathbb{R}$, where $\boldsymbol{\pi}^\circ$ is the stationary distribution associated to the underlying Markov chain of $(X,\mathbb{P}^\circ_x)$, $x\in\mathbb{R}\backslash\{0\}$, $\boldsymbol{\kappa}^\circ$ is the matrix exponent of the ascending ladder MAP of $(X,\mathbb{P}^\circ_x)$, $x\in\mathbb{R}\backslash\{0\}$ and $\hat{\boldsymbol{\kappa}}^{\circ}$ is that of its  dual. It is easily verified that 
\[
\boldsymbol{\pi}^\circ\propto\left[\begin{array}{c} \sin(\pi\alpha\rhohat)\\ \sin(\pi\alpha\rho)\end{array}\right]
\]
and hence, without loss of generality we may assume that $\boldsymbol{\pi}^\circ = \boldsymbol{\upsilon}(\alpha-1)$.  Inserting the Wiener--Hopf factorisation into the expression on the right-hand side  of (\ref{2WHFs}), in a straightforward fashion, one readily deduces that, for $\lambda\geq 0$
\begin{equation}
\boldsymbol{\kappa}(\lambda) = \boldsymbol{\Delta}_{\boldsymbol{\pi}^\circ}\boldsymbol{\kappa}^\circ(\lambda+\alpha -1) \boldsymbol{\Delta}_{\boldsymbol{\pi}^\circ}^{-1}.
\label{kappashift}
\end{equation}
This can also be verified by directly performing the Esscher transform to the process $(H, J^+)$ through an  application of Proposition \ref{p:mg and com} at the stopping time $L^{-1}_t$, $t>0$.
The reader should be careful to note that an additive shift of $\alpha-1$ in the argument of $\boldsymbol{F}$, the exponent of $(\xi, J)$, corresponds to an additive shift of $-(\alpha-1)$ in the argument of $\boldsymbol{\kappa}$ as the exponent of $(H, J^+)$ is given by $-\boldsymbol{\kappa}(-z)$.

Thanks to (\ref{kappashift}), it therefore follows that, to know $\boldsymbol{\kappa}(\lambda)$, it suffices to compute $\boldsymbol{\kappa}^\circ(\lambda)$. This is favourable on account of the fact that the origin is polar for $(X,\mathbb{P}^\circ_x)$, $x\in\mathbb{R}\backslash\{0\}$, which implies that the MAP corresponding to $\boldsymbol{F}^\circ(z)$ does not drift to $-\infty$ and hence the conditions of Lemma \ref{MRT}  are met.
As we shall soon see, in dealing with $\boldsymbol{\kappa}^\circ(\lambda)$ via this lemma, we shall make effective use of the Riesz--Bogdan--Zak transform. Using obvious notation, let us write
\[
\boldsymbol{\kappa}^\circ(\lambda)
=\text{diag}(\Phi^\circ_1(\lambda), \Phi^\circ_2(\lambda)) - \boldsymbol{\Lambda}^\circ\circ\boldsymbol{K}^\circ(\lambda), \qquad \lambda \geq 0.
\]
The analogue of Lemma \ref{PhiLemma}, but now for $\boldsymbol{\kappa}^\circ$, reads as follows.

\begin{lemma}
We have 
\[
\Phi^\circ_1(\lambda) =\frac{\sin(\pi\alpha\rho) + \sin(\pi\alpha\rhohat)}{\sin(\pi\alpha\rhohat)}{c}(\alpha)\phi_{\alpha\rho +1, \alpha\rhohat}(\lambda), \qquad \lambda \geq 0.
\]
\end{lemma}

\begin{proof} Referring to (\ref{decompose}) and the discussion leading to (\ref{dissleadto}), we have, for $u>0$, 
\begin{eqnarray}
\lefteqn{\pi_1^\circ{n}^\circ_1(\epsilon({\zeta}) > u, J(\zeta) =1, \zeta<\infty)}&&\notag\\
&& =-\frac{\d }{\d u}\lim_{x\to0} \int_0^1{\mathbb{P}}^\circ_x\left(X_{\tau^+_1}>{\rm e}^u, \overline{X}_{\tau^+_1- }\in \d z, \,|\underline{X}_{\tau^+_1-}|<z, \tau^+_1<\tau^-_{-1} \right)\notag\\
&& =-\frac{\d }{\d u}\int_0^1\lim_{x\to0}{\mathbb{P}}^\circ_x\left(X_{\tau^+_1}>{\rm e}^u, \overline{X}_{\tau^+_1- }\in \d z, \tau^+_1<\tau^-_{-z}  \right)\notag\\
&& =-\frac{\d }{\d u}\int_0^1\lim_{x\to0}\left.\frac{\rm d}{{\rm d} y}{\mathbb{P}}^\circ_x\left(X_{\tau^+_1}>{\rm e}^u, \overline{X}_{\tau^+_1- }\leq y, \tau^+_1<\tau^-_{-z}  \right)\right|_{y=z} \d z.
\label{insert}
\end{eqnarray}
Moreover, defining, for $a<b$,
\[
\tau^{(a,b)} = \inf\{t>0 : X_t \in (a, b)\},
\]
and writing $\hat{\mathbb{P}}_x$, $x\in\mathbb{R}$, for the probabilities of $-X$ (the negative of a stable process),
 we have, for $0<x<y<1$ and $u>0$,
\begin{eqnarray}
\lefteqn{-\frac{\d}{\d u }\lim_{x\to0}{\mathbb{P}}^\circ_x\left(X_{\tau^+_1}>{\rm e}^u, \overline{X}_{\tau^+_1- }\leq y, \tau^+_1<\tau^-_{-z}  \right)}&&\notag\\
&&=-\frac{\d}{\d u }\lim_{x\to0}{\mathbb{P}}^\circ_x\left(X_{\tau^+_y}>{\rm e}^u,  \tau^+_y<\tau^-_{-z}  \right)\notag\\
&&=-\frac{\d}{\d u }\lim_{x\to0}{\mathbb P}_{-1/x}(X_{\tau^{(-1/y,1/z)} }  \in (-{\rm e}^{-u},0))\notag\\
&&=-\frac{\d}{\d u }\lim_{x\to0}\hat{\mathbb P}_{1/x}(X_{\tau^{(-1/z,1/y)} }  \in (0,{\rm e}^{-u}))\notag\\
&&= \hat{p}_{\pm\infty}\left(\frac{2yz {\rm e}^{-u} -z+y}{y+z}\right)\frac{2yz}{y+z}{\rm e}^{-u},
\label{rhsabove}
\end{eqnarray}
where, momentarily, we will  assume the limit, 
\[
\hat{p}_{\pm\infty}(\theta): = \lim_{|z|\to\infty} \hat{\mathbb{P}}_z(X_{\tau^{(-1,1)}} \in \d \theta)/ \d  \theta, \qquad \theta\in[-1,1],
\]
 exists. 
In order to continue the process of identifying an explicit expression for (\ref{rhsabove}), we appeal a distributional identity  found in \cite{KPW} which gives the law of the position of a stable process when it first enters a finite interval.
We can, in principle, take limits in the expression for  $\hat{\mathbb{P}}_z(X_{\tau^{(-1,1)}} \in \d y)$,  in particular, showing that $\hat{p}_{\pm}$ is well defined.
It turns out to be more convenient to root deeper into the proof of Theorem 1.1. of \cite{KPW} and fish out an alternative expression. From equation (20) in \cite{KPW} we have, for $y \in (-1,1)$ and  $\alpha \in (1,2)$,
\begin{eqnarray}
\hat{p}_{+\infty}(y) &:=& \lim_{z\to\infty}\hat{\mathbb{P}}_z(X_{\tau^{(-1,1)}} \in \d y)/ \d  y
\notag\\
  &&= \frac{\sin(\pi\alpha\rho)}{\pi}
    (1+y)^{-\alpha\rhohat}
    (1-y)^{-\alpha\rho} \notag\\
  &&\hspace{1cm}\times \lim_{z\to\infty}
    \biggl[
      (y+1)
      (z-1)^{\alpha\rho}
      (z+1)^{\alpha\rhohat-1}
      (z-y)^{-1} \biggr.  \notag\\
&&  \hspace{2cm} + (1-\alpha\rhohat ) 2^{\alpha-1}
      \int_0^{\frac{z-1}{z+1}} t^{\alpha\rho-1} (1-t)^{1-\alpha} \, \dd t
    \biggr]\notag\\
    &&=(1-\alpha\rhohat) 2^{\alpha-1}\frac{\sin(\pi\alpha\rho)}{\pi}
    (1+y)^{-\alpha\rhohat}
    (1-y)^{-\alpha\rho} \int_0^{1} t^{\alpha\rho-1} (1-t)^{1-\alpha} \, \dd t\notag\\
    &&= 2^{\alpha-1}(1-\alpha\rhohat) \frac{\Gamma(\alpha\rho)\Gamma(2-\alpha)}{\Gamma(2-\alpha\rhohat)\Gamma(\alpha\rho)\Gamma(1-\alpha\rho)}
    (1+y)^{-\alpha\rhohat}
    (1-y)^{-\alpha\rho}\notag\\
    &&= {c}(\alpha) 
    (1+y)^{-\alpha\rhohat}
    (1-y)^{-\alpha\rho},
    \label{stationary_inshoot}
\end{eqnarray}
where 
\[
{c}(\alpha) =2^{\alpha-1}\frac{\Gamma(2-\alpha)}{\Gamma(1-\alpha\rhohat)\Gamma(1-\alpha\rho)}.
\]
By appealing to duality, one also easily verifies in a similar fashion that 
\[
\hat{p}_{-\infty}(y) : = \lim_{z\to-\infty}\hat{\mathbb{P}}_{z}(X_{\tau^{(-1,1)}}\in \d y)/\d y = p_{\infty}(-y) =  \hat{p}_{+\infty}(y),
\]
 where the function $p_\infty$ is the same as $\hat{p}_\infty$ albeit the roles of $\rho$ and $\rhohat$ are interchanged. Hence we may accordingly refer to the function 
\begin{equation}
\hat{p}_{\pm\infty}(y):=\hat{p}_{+\infty}(y) = \hat{p}_{-\infty}(y)= {c}(\alpha) 
    (1+y)^{-\alpha\rhohat}
    (1-y)^{-\alpha\rho}.
    \label{stinshoot}
\end{equation}

Returning to (\ref{insert}) and (\ref{rhsabove}),
we therefore have
\begin{eqnarray*}
\pi^\circ_1{n}^\circ_1(\epsilon({\zeta}) > u, J(\zeta) =1, \zeta<\infty) = {c}(\alpha)\int_0^1 \hat{\beta}(z {\rm e}^{-u}){\rm e}^{-u}\d z,
\end{eqnarray*}
where $\hat{\beta}(\theta) = \{\hat{p}'_{\pm\infty}(\theta)(\theta + 1) + \hat{p}_{\pm\infty}(\theta)\}/2{c}(\alpha)$, which can easily be verified to satisfy 
\begin{equation*}
\hat{\beta}(\theta) = \frac{\alpha\rho}{(1-\theta)^{\alpha\rho + 1}(1+\theta)^{\alpha\rhohat}} - \frac{(\alpha-1)/2}{(1-\theta)^{\alpha\rho}(1+\theta)^{\alpha\rhohat}},
\end{equation*}
for $\theta\in[-1,1]$.
We can now compute 
\begin{eqnarray*}
\pi^\circ_1{\Phi}^\circ_1(\lambda) &=&{c}(\alpha)\lambda \int_0^\infty \int_0^1 {\rm e}^{-(\lambda+1) u} \hat{\beta}(z {\rm e}^{-u})
\d z\ \d u\\
&=&{c}(\alpha)\int_0^\infty (1-{\rm e}^{-\lambda w})\hat{\beta} ({\rm e}^{-w}){\rm e}^{-w}\d w, \qquad \lambda\geq 0,
\end{eqnarray*}
as required. Note that the second equality condenses two changes of variable and an application of Fubini's theorem, similar in spirit to earlier computations, into one step. The details are straightforward and left to the reader.
\end{proof}

\begin{proof}[Proof of Theorem \ref{WHFMAPHG>1} (i)]
Let ${\mathbf{P}}^\circ_{0,i}$ be the law of the MAP whose matrix exponent is $\boldsymbol{F}^\circ$, issued from $(0,i)$. Following previous reasoning, we compute, 
\begin{eqnarray}
\lim_{a\to\infty }{\mathbf{P}}^\circ_{0,i}(H^{+}(T_{a})-a\leq u  ;J^{+}(T_{a})=1 )&=
&\lim_{x\to0}{\mathbb{P}}^\circ_x(X_{\tau^+_1} \leq {\rm e}^u; \tau^+_1<\tau^-_{-1})\notag\\
&=&\lim_{x\to0}\hat{\mathbb P}_{1/x}(X_{\tau^{(-1,1)} }  \in ({\rm e}^{-u},1))\notag\\
&=&\int_{{\rm e}^{-u}}^1\hat{p}_{\pm\infty}(y)\d y.\label{rhoexchangelater}
\end{eqnarray}
Hence, for $\lambda\geq 0$,
\begin{eqnarray*}
\Theta^\circ_1(\lambda)  &:=&\lim_{a\to\infty }{\mathbf{P}}^\circ_{0,i}({\rm e}^{-\lambda(H^{+}(T_{a})-a)} ;J^{+}(T_{a})=1)=
\int_0^\infty {\rm e}^{-(\lambda+1) u} \hat{p}_{\pm\infty}({\rm e}^{-u}) \d u,
\end{eqnarray*} 
and so, again appealing to obvious notation, from Lemma \ref{MRT},
\begin{eqnarray*}
\pi_2^\circ\Lambda^\circ_{2,1}{\boldsymbol{K}}^\circ(\lambda)_{2,1}
&=&\Theta^\circ_1(\lambda)- {\pi^\circ_1}\frac{\Phi^\circ_1(\lambda)}{\lambda}\\
&=&\int_0^\infty \frac{\left(1-{\rm e}^{-\lambda u}\right)}{\lambda}\left\{\hat{p}_{\pm\infty}'({\rm e}^{-u}){\rm e}^{-u} + \hat{p}_{\pm\infty}({\rm e}^{-u}) \right\}{\rm e}^{-u}\d u\\
&&-{c}(\alpha)\int_0^\infty \frac{(1-{\rm e}^{-\lambda u})}{\lambda}\hat{\beta} ({\rm e}^{-u}){\rm e}^{-u}\d u\\
&=&\int_0^\infty\frac{\left(1-{\rm e}^{-\lambda u}\right)}{\lambda} \frac{1}{2}\left\{\hat{p}'_{\pm\infty}({\rm e}^{-u})({\rm e}^{-u} - 1) +\hat{p}_{\pm\infty}({\rm e}^{-u}) \right\}{\rm e}^{-u}\d u\\
&=& {c}(\alpha)\frac{\phi_{\alpha\rho,\alpha\rhohat+1}(\lambda)}{\lambda}, \qquad \lambda\geq 0.
\end{eqnarray*}
This tells us that up to the multiplicative constant ${c}(\alpha)(\sin(\pi\alpha\rho) + \sin(\pi\alpha\rhohat))/{\sin(\pi\alpha\rho)\sin(\pi\alpha\rhohat)}$, for $\lambda\geq 0$,
\[
\Lambda^\circ_{2,1}{\boldsymbol{K}}^\circ(\lambda)_{2,1} = \sin(\pi\alpha\rhohat)\frac{\phi_{\alpha\rho,\alpha\rhohat+1}(\lambda)}{\lambda}\quad\text{ and }\quad\Lambda^\circ_{2,1} =\sin(\pi\alpha\rhohat){\phi_{\alpha\rho,\alpha\rhohat+1}'(0+)}.
\]
By exchanging the roles of $\rho$ and $\rhohat$, we now have enough identities to complete fill out the entries of $\boldsymbol{\kappa}(\lambda)$.
\end{proof}

\section{The ascending ladder MAP for the dual process}\label{updual}

Using (\ref{dualF}) and (\ref{pi}) a straightforward computation gives us
\begin{eqnarray*}
\hat{\boldsymbol{F}}(z)
&=& \left[
  \begin{array}{cc}
    - \dfrac{\Gamma(\alpha+z)\Gamma(1-z)}
      {\Gamma(\alpha\hat\rho+z)\Gamma(1-\alpha\hat\rho- z)}
    & \dfrac{\Gamma(\alpha+z)\Gamma(1- z)}
      {\Gamma(\alpha\hat\rho)\Gamma(1-\alpha\hat\rho)}
    \\
    &\\
    \dfrac{\Gamma(\alpha+z)\Gamma(1-z)}
      {\Gamma(\alpha\rho)\Gamma(1-\alpha\rho)}
        & - \dfrac{\Gamma(\alpha+z)\Gamma(1-z)}
      {\Gamma(\alpha\rho+z)\Gamma(1-\alpha\rho-z)}
  \end{array} 
  \right],
\end{eqnarray*}
which is well defined for $\Re(z)\in(-\alpha, 1)$. Glancing back, one quickly realises that  $\hat{\boldsymbol{F}}(z)$ takes the form of  $\boldsymbol{F}_{\alpha-1}(z)$ (from the Proof of Theorem \ref{th:BZ}), but with the roles of $\rho$ and $\hat\rho$ interchanged; otherwise written
\begin{equation}
\hat{\boldsymbol{F}}(z)  = \left.\boldsymbol{F}_{\alpha-1}(z)\right|_{\rho\leftrightarrow\rhohat} = \left.\boldsymbol{F}^\circ(z)\right|_{\rho\leftrightarrow\rhohat}
\label{exchange}
\end{equation}
for $\Re(z)\in(-\alpha,1).$
This tells us that the dual of $(\xi, J)$ is the MAP which underlies the rssMp $(X, \hat{\mathbb{P}}^\circ_x)$, $x\in\mathbb{R}\backslash\{0\}$ through the Lamperti--Kiu transform.

We may now proceed to calculate $\hat{\boldsymbol{\kappa}}(\lambda)$ by again taking advantage of the Esscher transform to remove the effects of killing. Indeed, referring to (\ref{exchange}) and noting that, without loss of generality, we may take 
\[
\boldsymbol{v}(\alpha-1)|_{\rho\leftrightarrow\rhohat} = \boldsymbol{\pi},
\]
 it is immediately clear that 
\[
\hat{\boldsymbol{\kappa}}(\lambda) = \boldsymbol{\Delta}_{\boldsymbol \pi}^{-1}\left.\boldsymbol{\kappa}(\lambda+1-\alpha)\right|_{\rho\leftrightarrow\rhohat}\boldsymbol{\Delta}_{\boldsymbol \pi}, \qquad \lambda\geq0.
\]
Referring to (\ref{kappashift}), this means, in particular that 
\begin{equation}
\hat{\boldsymbol{\kappa}}(\lambda) = \left. \boldsymbol{\kappa}^\circ(\lambda)\right|_{\rho\leftrightarrow\rhohat}, \qquad \lambda \geq 0.
\label{preremark}
\end{equation}
That is to say  $\hat{\boldsymbol{\kappa}}(\lambda)$ is nothing more than $\boldsymbol{\kappa}^\circ(\lambda)$ albeit with the roles of $\rho$ and $\rhohat$ interchanged. 

\begin{remark}\rm\label{absorbedat0}
Whist (\ref{preremark}) is helpful for computing $\hat{\boldsymbol{\kappa}}$ in the case that $\alpha\in(1,2)$, it also tells us something about $\boldsymbol{\kappa}^\circ$ in the case that $\alpha\in(0,1)$, where $\boldsymbol{\kappa}^\circ$ has not already been evaluated. For the latter regime of $\alpha$, appealing to  Theorem \ref{WHFMAPHG} (ii), we note  that  $\hat{\boldsymbol{\kappa}}(0)_{1,1}>- \hat{\boldsymbol{\kappa}}(0)_{1,2}$ and, similarly, $\hat{\boldsymbol{\kappa}}(0)_{2,2}>- \hat{\boldsymbol{\kappa}}(0)_{2,1}$. This implies that the ascending ladder MAP associated to $(X, \hat{\mathbb{P}}^\circ_x)$, $x\in\mathbb{R}\backslash\{0\}$, and hence $(X, {\mathbb P}^\circ_x)$, $x\in\mathbb{R}\backslash\{0\}$, is subject to killing. In turn, this means   that, when $\alpha\in(0,1)$, the MAP associated to $\boldsymbol{F}^\circ(z)$ drifts to $-\infty$ and, accordingly,  $(X, {\mathbb P}^\circ_x)$, $x\in\mathbb{R}\backslash\{0\}$, is a real-valued self-similar Markov process which experiences absorption at the origin.

In this sense, when $\alpha\in(0,1)$, $(X, {\mathbb P}^\circ_x)$, $x\in\mathbb{R}\backslash\{0\}$ may be reasonably named the stable process conditioned to be continuously absorbed at the origin. Indeed, in further work following ideas of \cite{CPR}, we hope to give more mathematical substance to this remark.
\end{remark}

\section{Cram\'er-type results for Lamperti-stable MAPs}\label{cramer}


A classical computation that emerges from the Wiener--Hopf factorisation in the setting of a L\'evy process that drifts to $-\infty$, which also has  a non-zero root of its characteristic exponent, is Cram\'er's asymptotic estimate for the probability of first passage above a threshold, as well as the asymptotic conditional overshoot distribution conditional of first passage. See for example \cite{BD1994}. In the current setting we have noted that when $\alpha\in(1,2)$ and when $\alpha\in(0,1)$ the  ascending ladder processes of the Lamperti-stable MAP and and the ascending ladder processes of the dual Lamperti-stable MAP, respectively, undergo killing. Moreover, we have also noted the existence of roots to the leading eigenvalue of the associated matrix exponent. 
This means that we can expect to see Cram\'er-type results in each of these regimes. In this respect, we have two main theorems in this section.

\begin{theorem}\label{Cramer>1} When $\alpha\in(0,1]$, $\mathbf{P}_{0,i}(T_a<\infty )=1$,  for $i=1,2$,
where as, when $\alpha\in(1,2)$,
\[
\lim_{a\to\infty}{\rm e}^{(\alpha - 1)a}\mathbf{P}_{0,i}(T_a<\infty ) = 2^{\alpha-1}\frac{\Gamma(2-\alpha)}{\Gamma(1-\alpha\rhohat)\Gamma(1-\alpha\rho)
}\left\{\frac{\pi}{\sin(\pi\alpha\rho)}\mathbf{1}_{(i=1)} + \frac{\pi}{\sin(\pi\alpha\rhohat)}\mathbf{1}_{(i=2)}\right\}.
\]
Moreover, when $\alpha\in(0,2)$, for $i = 1,2$ and $u>0$,
\begin{eqnarray*}
\lefteqn{\lim_{a\to\infty }\mathbf{P}_{0,i}(H^+T_{a})-a\in \d u  ;J^{+}(T_{a})=j | T_a<\infty)}&&\\
&&=\left\{\begin{array}{ll}
\dfrac{\sin(\pi\alpha\rho)}{\pi}{\rm e}^{-\alpha u}   (1+{\rm e}^{-u})^{-\alpha\rhohat}
    (1-{\rm e}^{-u})^{-\alpha\rho}\d u &\text{ if } j = 1,\\
    &\\
    \dfrac{\sin(\pi\alpha\rhohat)}{\pi}{\rm e}^{-\alpha u}   (1+{\rm e}^{-u})^{-\alpha\rho}
    (1-{\rm e}^{-u})^{-\alpha\rhohat}\d u & \text{ if }j =2.
        \end{array}
    \right.
    \end{eqnarray*}
\end{theorem}
\begin{proof} The first claim follows by virtue of the fact that, as noted in Section \ref{smalla}, the ascending ladder MAP $(H, J^+)$ experiences no killing when $\alpha\in(0,1]$.
Recall that 
\[
\tau^{(-1,1)} := \inf\{t>0: X_t \in (-1,1)\}.
\]
Thanks to (\ref{scaleit}) and Theorem \ref{th:BZ}, 
\begin{eqnarray}
\lefteqn{\mathbf{P}_{0,1}(T_a<\infty)
}&&\hspace{-0.5cm}\notag\\
&&\hspace{-0.5cm}=\mathbb{P}_{{\rm e}^{-a}}(\tau^+_{1} \wedge \tau^-_{-1}<\tau^{\{0\}})\notag\\
&&\hspace{-0.5cm}=\mathbb{P}^\circ_{-{\rm e}^{a}}(\tau^{(-1,1)}<\infty)\notag\\
&&\hspace{-0.5cm}=\hat{\mathbb{P}}^\circ_{{\rm e}^{a}}(\tau^{(-1,1)}<\infty)\label{notatend}\\
&&\hspace{-0.5cm}= 
\frac{ {\rm e}^{-(\alpha -1)a}}{2\sin(\pi\alpha\rho)}\hat{\mathbb{E}}_{{\rm e}^{a}}\left(
\left(2\sin(\pi\alpha\rho)\mathbf{1}_{(X_{\tau^{(-1,1)}}>0)} +2\sin(\pi\alpha\hat\rho)\mathbf{1}_{(X_{\tau^{(-1,1)}}<0)}\right)
| X_{\tau^{(-1,1)}}|^{\alpha -1}
\right),\notag
\end{eqnarray}
where, for each $w\in\mathbb{R}\backslash\{0\}$, $\hat{\mathbb{P}}^\circ_{w}$  plays the role of $\mathbb{P}^\circ_w$ with  $\rho$ and $\hat\rho$  interchanged (i.e. it plays the role of $\mathbb{P}^\circ_w$  for $-X$).
Recalling the definition of the limiting distribution $\hat{p}_{\pm\infty}$ given in (\ref{stinshoot}), we thus have 
\begin{eqnarray}
\lefteqn{\lim_{a\to\infty}{\sin(\pi\alpha\rho)}{\rm e}^{(\alpha -1)a}\mathbf{P}_{0,1}(T_a<\infty)}&&\notag\\
&&={c}(\alpha)\sin(\pi\alpha\rho)
    \int_0^1 y^{\alpha -1}(1+y)^{-\alpha\rhohat}
    (1-y)^{-\alpha\rho} \d y\notag\\
   && \hspace{1cm}+{c}(\alpha)\sin(\pi\alpha\rhohat)
   \int_0^1  y^{\alpha -1}(1-y)^{-\alpha\rhohat}
    (1+y)^{-\alpha\rho} \d y\notag\\
    &&{c}(\alpha)\pi\frac{\sin(\pi\alpha\rho)}{\pi}\frac{\Gamma(\alpha)\Gamma(1-\alpha\rho)}{\Gamma(1+\alpha\rhohat)}{_{2}}\mathcal{F}_1(\alpha\rhohat,\alpha,
    \alpha\rhohat+1 ; -1)\notag\\
    &&\hspace{1cm}+{c}(\alpha)\pi\frac{\sin(\pi\alpha\rhohat)}{\pi}\frac{\Gamma(\alpha)\Gamma(1-\alpha\rhohat)}{\Gamma(1+\alpha\rho)}{_{2}}\mathcal{F}_1(\alpha\rho,\alpha,
    \alpha\rho+1 ; -1)\notag\\
    &&={c}(\alpha)\pi,
    \label{alreadybeendealtwith}
    \end{eqnarray}
where ${_2}\mathcal{F}_1(a,b,c; z)$ is the usual hypergeometric function and the final equality is a remarkable simplification which follows from one of the many identities for the aforesaid functions. See for example the first formula at the \texttt{functions.wolfram.com} webpages \cite{wolfram}. If, on the left-hand side of (\ref{alreadybeendealtwith}), we replace $\mathbf{P}_{0,1}$ by $\mathbf{P}_{0,2}$, the only thing that changes in the statement is that we must replace $\sin(\pi\alpha\rho)$ by $\sin(\pi\alpha\rhohat)$ on the left-hand side. This completes the proof of the first part of the theorem.

\bigskip

For the next part, we split the proof into the cases that $\alpha\in(0,1]$ and $\alpha\in(1,2)$. 
In the former case the result was already established in (\ref{lastsection}).
For the latter case,  appealing again to the Riesz--Bogdan--Zak transform, (\ref{notatend}), (\ref{alreadybeendealtwith}) and (\ref{stationary_inshoot}), we have, for $u>0$, 
\begin{eqnarray}
\lefteqn{\lim_{a\to\infty }\mathbf{P}_{0,1}(H^+T_{a})-a\leq u  ;J^{+}(T_{a})=1 | T_a<\infty)}&&\notag\\
&&=\lim_{a\to\infty}\frac{\mathbb{P}_{{\rm e}^{-a}}(X_{\tau^+_1} \leq {\rm e}^u; \tau^+_1<\tau^-_{-1}\wedge\tau^{\{0\}})}{\mathbb{P}_{{\rm e}^{-a}}(\tau^+_1\wedge\tau^-_{-1}<\tau^{\{0\}})}\notag\\
&&=\lim_{a\to\infty}\frac{\hat{\mathbb{P}}^\circ_{{\rm e}^{a}}(X_{\tau^{(-1,1)}} \in ({\rm e}^{-u} ,1 ) \,; \,\tau^{(-1,1)}<\infty
)}{\hat{\mathbb{P}}^\circ_{{\rm e}^{a}}(\tau^{(-1,1)}<\infty)}\notag\\
&&=\frac{\sin(\pi\alpha\rho)\int_{{\rm e}^u}^1 \theta^{\alpha-1}\hat{p}_{\pm\infty}(\theta)\d \theta}{{c}(\alpha)\pi}\notag\\
&&=\frac{\sin(\pi\alpha\rho)}{ \pi}
\int_{{\rm e}^{-u}}^1\theta^{\alpha -1}   (1+\theta)^{-\alpha\rhohat}
    (1-\theta)^{-\alpha\rho}\d \theta, \qquad u\geq 0, 
    \label{**}
\end{eqnarray}
which is equivalent to the statement in the second part of the theorem when $i=1$. It turns out that the asymptotic is  unaffected when $i=2$, however the details are left to the reader to verify. We also leave it as an exercise for the reader to check that when the event $\{J^+(T_a) = 1\}$ is replaced by $\{J^+(T_a) = 2\}$ on the left-hand side of (\ref{**}), the resulting asymptotic is the same but with the roles of $\rho$ and $\rhohat$ interchanged.
\end{proof}

It is worth noting from the proof of this theorem that the methodology allows us access to new identities for stable processes with $\alpha\in(1,2)$.  For example, the following polynomial asymptotic decay for the probability that the stable processes escapes $(-1,1)$ on before hitting the origin.
\begin{cor}For $\alpha\in(1,2)$,
\[
\lim_{x\to0}s(x) x^{1-\alpha}\mathbb{P}_x (\tau^+_1\wedge\tau^-_{-1}<\tau^{\{0\}}) =2^{\alpha-1}\frac{\Gamma(2-\alpha)}{\Gamma(1-\alpha\rhohat)\Gamma(1-\alpha\rho),
}
\]
where 
\[
s(x) := \frac{\sin(\pi\alpha\rho)}{\pi}\mathbf{1}_{(x>0)} + \frac{\sin(\pi\alpha\rhohat)}{\pi}\mathbf{1}_{(x<0)}.
\]
\end{cor}
Another  example of a new fluctuation result for stable processes is captured in the corollary immediately below. 
\begin{cor}
For $\alpha\in(1,2)$, $\theta>0$ and $x\in(0,1)$,
\begin{eqnarray*}
\lefteqn{\mathbb{P}_x(X_{\tau^+_1}-1 \in \d \theta,  \tau^+_1<\tau^-_{-1} \wedge
\tau^{\{0\}})}&&\notag\\
&&=\frac{\sin(\pi\alpha\rho)}{\pi}
(1+x)^{\alpha\rhohat}(1-x)^{\alpha\rho}(2+\theta)^{-\alpha\rhohat}\theta^{-\alpha\rho}( 1+\theta-x)^{-1}\\
&&\hspace{1cm} -(\alpha-1) \frac{\sin(\pi\alpha\rho)}{\pi}
(2+\theta)^{-\alpha\rhohat}\theta^{-\alpha\rho} (1+\theta)^{-1}x^{\alpha-1} \int_1^{1/x} (t-1)^{\alpha\rho-1} (t+1)^{\alpha\rhohat-1}\, \d t.
\end{eqnarray*}
\end{cor}
\begin{proof}
 According to Theorem \ref{th:BZ}, we thus have 
\begin{eqnarray*}
\lefteqn{\mathbb{P}_x(X_{\tau^+_1}-1>\theta, \tau^+_1<\tau^-_{-1} \wedge
\tau^{\{0\}})}&&\notag\\
&&=\mathbb{P}_x(-1/X_{\tau^+_1}\in(-(1+\theta)^{-1}, 0), \tau^+_1<\tau^-_{-1} \wedge
\tau^{\{0\}})\notag\\
&&=\mathbb{P}^\circ_{-1/x}(X_{\tau^{(-1,1)}}\in (-(1+\theta)^{-1}, 0),    \tau^{(-1,1)} <\infty)\notag\\
&&=x^{\alpha-1}\hat{\mathbb{E}}_{1/x}\Big(|X_{\tau^{(-1,1)}}|^{\alpha -1} \, ;\, X_{\tau^{(-1,1)}}\in (0, (1+\theta)^{-1})\Big),
\end{eqnarray*}
for $\theta>0$. It follows that 
\[
\mathbb{P}_x(X_{\tau^+_1}-1 \in \d \theta,  \tau^+_1<\tau^-_{-1} \wedge
\tau^{\{0\}}) =x^{\alpha-1} \left(\frac{1}{1+\theta}\right)^{\alpha+1}\hat{p}_{1/x}\left(\frac{1}{1+\theta}\right)\d \theta,
\]
where, for $x>1$ and $y\in[-1,1]$, $\hat{p}_x(y) = \hat{\mathbb{P}}_x(X_{\tau^{(-1,1)}} \in \d y)/\d y$. The latter can be found in Theorem 1.1 of \cite{KPW} and is given by 
\begin{eqnarray*}
\hat{ \mathbb{P}}_x(X_{\tau_{-1}^1} \in \d y)/\d y
 &=& \frac{\sin(\pi\alpha\rho)}{\pi}
    (x+1)^{\alpha\rhohat}
    (x-1)^{\alpha\rho}
    (1+y)^{-\alpha\rhohat}
    (1-y)^{-\alpha\rho}
    (x-y)^{-1} \\
  && {} - (\alpha-1)
    \frac{\sin(\pi\alpha\rho)}{\pi}
    (1+y)^{-\alpha\rhohat}
    (1-y)^{-\alpha\rho}
    \int_1^x (t-1)^{\alpha\rho-1} (t+1)^{\alpha\rhohat-1}\, \d t.
\end{eqnarray*}
The result now follows by straightforward algebra.
\end{proof}

\bigskip

Let us now turn our attention to the dual of the Lamperti-stable MAP when $\alpha\in(0,1)$. We denote its law by $\hat{\mathbf{P}}_{x,i}$, for $x\in\mathbb{R}$ and $i=1,2$. Recall from Section \ref{updual} that this MAP corresponds to the rssMp $(X, \hat{\mathbb{P}}^\circ_x)$, $x\in\mathbb{R}\backslash\{0\}$.
The analogue of Theorem \ref{Cramer>1} takes the following form.
\begin{theorem}\label{Cramer<1}If $\alpha\in(0,1)$,
then, for $i =1,2$, 
\[
\lim_{a\to\infty}{\rm e}^{(1-\alpha)a}\hat{\mathbf{P}}_{0,i}(T_a<\infty)  = 
\frac{2^{1-\alpha}}{\Gamma(\alpha\rho)\Gamma(\alpha\rhohat)
\Gamma(2-\alpha)}\left\{\frac{\pi}{\sin(\pi\alpha\rho)}\mathbf{1}_{(i=1)} + \frac{\pi}{\sin(\pi\alpha\rhohat)}\mathbf{1}_{(i=2)}\right\}.
\]
If $\alpha\in [1,2)$, then $\hat{\mathbf{P}}_{0,i}(T_a<\infty)  = 1$, for $i=1,2$. Moreover,  
for $\alpha\in(0,2)$, $i = 1,2$ and $u>0$,
\begin{eqnarray*}
\lefteqn{\lim_{a\to\infty }\hat{\mathbf{P}}_{0,i}(H^+T_{a})-a\in \d u  ;J^{+}(T_{a})=j | T_a<\infty)}&&\\
&&=\left\{
\begin{array}{ll}
 2^{\alpha-1}\dfrac{\Gamma(2-\alpha)}{\Gamma(1-\alpha\rhohat)\Gamma(1-\alpha\rho)}
{\rm e}^{- u}   (1+{\rm e}^{-u})^{-\alpha\rho}
    (1-{\rm e}^{-u})^{-\alpha\rhohat}
    \d u & 
   \text{ if } j =1\\
   &\\
   2^{\alpha-1}\dfrac{\Gamma(2-\alpha)}{\Gamma(1-\alpha\rhohat)\Gamma(1-\alpha\rho)}
{\rm e}^{- u}   (1+{\rm e}^{-u})^{-\alpha\rhohat}
    (1-{\rm e}^{-u})^{-\alpha\rho}
    \d u &
   \text{ if } j =2.
    \end{array}
    \right.
    \end{eqnarray*} 
\end{theorem}
\begin{proof} Appealing to (\ref{scaleit}), we start by noting that
\[
\hat{\mathbf{P}}_{0,1}(T_a<\infty) = \hat{\mathbb{P}}^\circ_{{\rm e}^{-a}}(\tau^+_1\wedge\tau^-_{-1}<\infty),
\]
and hence, by again making use of Theorem \ref{2sidedexit} and Theorem \ref{th:BZ}, we also have that, when $\alpha\in(0,1)$, 
\begin{eqnarray}
\lefteqn{\lim_{a\to\infty}
\sin(\pi\alpha\rho){\rm e}^{(1-\alpha)a}\hat{\mathbf{P}}_{0,1}(T_a<\infty) }&&\notag\\
&&=\lim_{a\to\infty}\hat{\mathbb{E}}_{{\rm e}^{-a}}\left(
\left(\sin(\pi\alpha\rho)\mathbf{1}_{(X_{\tau^+_1\wedge\tau^-_{-1}} >0)}+ \sin(\pi\alpha\rhohat)\mathbf{1}_{(X_{\tau^+_1\wedge\tau^-_{-1}} <0)}\right) |X_{\tau^+_1\wedge\tau^-_{-1}}|^{\alpha-1} 
\right)\notag\\
&&=\frac{\sin(\pi\alpha\rhohat)\sin(\pi\alpha\rho)}{\pi} \int_0^\infty {\rm e}^{(\alpha -1)u}({\rm e}^u-1)^{-\alpha\rhohat}({\rm e}^u+1)^{-\alpha\rho} \d u\notag\\
&&\hspace{0.5cm}+\frac{\sin(\pi \alpha\rhohat)\sin(\pi\alpha\rho)}{\pi} \int_0^\infty {\rm e}^{(\alpha -1)u}({\rm e}^u-1)^{-\alpha\rho}({\rm e}^u+1)^{-\alpha\rhohat} \d u\notag\\
&&=\frac{\sin(\pi\alpha\rhohat)\sin(\pi\alpha\rho)}{\pi(1-\alpha\rhohat)(1-\alpha\rho)} (1-\alpha\rho) {_2}\mathcal{F}_1(1, \alpha\rho, 2-\alpha\rhohat; -1)\notag\\
&&\hspace{0.5cm}+\frac{\sin(\pi\alpha\rhohat)\sin(\pi\alpha\rho)}{\pi(1-\alpha\rhohat)(1-\alpha\rho)} (1-\alpha\rhohat) {_2}\mathcal{F}_1(1, \alpha\rhohat, 2-\alpha\rho; -1)\notag\\
&& = \frac{\sin(\pi\alpha\rhohat)\sin(\pi\alpha\rho)}{\pi(1-\alpha\rhohat)(1-\alpha\rho)} \times 2^{1-\alpha}\frac{\Gamma(2-\alpha\rho)\Gamma(2-\alpha\rhohat)}{\Gamma(2-\alpha)}\notag\\
&&=\frac{2^{1-\alpha}\pi }{\Gamma(\alpha\rho)\Gamma(\alpha\rhohat)
\Gamma(2-\alpha)}.
\label{hatconstant}
\end{eqnarray}
where the penultimate equality is again remarkably due to a very particular identity for hypergeometric functions; see the second formula in the \texttt{functions.wolfram.com} webpage \cite{wolfram}. If we repeat the computation with $\hat{\mathbf{P}}_{0,1}$ replaced by $\hat{\mathbf{P}}_{0,2}$, then the only other thing that changes in  (\ref{hatconstant}) is that $\sin(\pi\alpha\rho)$ is replaced by $\sin(\pi\alpha\rhohat)$ on the left-hand side. This completes the proof of the first part of the theorem. 

When $\alpha\in[1,2)$, the ascending ladder height MAP of the dual is not killed (see the discussion in in Section \ref{updual}) and hence $\hat{\mathbf{P}}_{0,i}(T_a<\infty)  = 1$, for $i=1,2$.

\bigskip

For the next part set $\alpha\in(0,1]$. Starting as we did in the proof of Theorem \ref{WHFMAPHG} (i), we note from Lemma \ref{MRT},  \eqref{hatconstant} and Theorem \ref{2sidedexit}  that
\begin{eqnarray*}
\lefteqn{\lim_{a\to\infty }\hat{\mathbf{P}}_{0,1}(H^{+}(T_{a})-a\leq u  ;J^{+}(T_{a})=1|T_a<\infty )}&&\\
&&=\lim_{a\to\infty}\hat{\mathbb{P}}^\circ_{{\rm e}^{-a}}(X_{\tau^+_1} \leq {\rm e}^u; \tau^+_1<\tau^-_{-1}|\tau^+_1\wedge\tau^-_{-1}<\infty)\\
&&=\lim_{a\to\infty}\frac{\hat{\mathbb{P}}^\circ_{{\rm e}^{-a}}(X_{\tau^+_1} -1\leq {\rm e}^u-1, \tau^+_1<\tau^-_{-1})}{\hat{\mathbb{P}}^\circ_{{\rm e}^{-a}}(\tau^+_1\wedge\tau^-_{-1}<\infty)}\\
&&=\frac{\Gamma(\alpha\rhohat)\Gamma(2-\alpha)}{\Gamma(1-\alpha\rho)}\hat{\mathbb{P}}(X_{\tau^+_1}^{\alpha -1}\,;\, X_{\tau^+_1} -1\leq {\rm e}^u-1, \tau^+_1<\tau^-_{-1})\\
&&=2^{\alpha-1}\frac{\Gamma(2-\alpha)}{\Gamma(1-\alpha\rhohat)\Gamma(1-\alpha\rho)}
\int_0^{{\rm e}^u-1} (\theta+1)^{\alpha-2} \theta^{-\alpha\rhohat} (\theta+2)^{-\alpha\rho} \d \theta.
\end{eqnarray*}
This is equivalent to the second statement of the theorem for $i=1$. The computation when $i=2$ can be performed similarly. Replacing the event $\{J^{+}(T_{a})=1\}$  by $\{J^{+}(T_{a})=2\}$ in the probability above, affects the final equality only exchanging the roles of $\rho$ and $\rhohat$. The details are left to the reader.

Finally, when $\alpha\in(1,2)$, the desired asymptotic can already be found in (\ref{rhoexchangelater}), as soon as one notes that $\hat{\mathbf{P}}_{0,i}$ agrees with $\mathbf{P}_{0,i}^\circ$ when the roles of $\rho$ and $\rhohat$ are exchanged.
\end{proof}

Similarly to before, one can proceed to extract further identities for the Doob $h$-transformed process $(X, \mathbb{P}^\circ_{x})$, $x\in\mathbb{R}\backslash\{0\}$, however, we leave this for the reader to amuse themselves with.

\section*{Acknowledgements}

The author would like to thank Loic Chaumont,  Alexey Kuznetsov, Victor Rivero and Weerapat Satitkanitkul for useful discussions. This work is sponsored by EPSRC grant EP/L002442/1.

%


\begin{thebibliography}{99}

\bibitem{wolfram} \texttt{http://functions.wolfram.com/HypergeometricFunctions/Hypergeometric2F1/}
\newblock\texttt{03/03/04/}

\bibitem{Alsmeyer} G. Alsmeyer :
 \newblock  On the Markov renewal theorem.
\newblock {\it Stoch. Proc. Appl.} {\bf 50}, 37-56, 1994.


\bibitem{Alsmeyer1} G. Alsmeyer :
 \newblock  Quasistochastic matrices and Markov renewal theory.
\newblock {\it J. Appl. Probab.} {\bf 51A}, 359 - 376, 2014.

\bibitem{AS} E. Arjas and T. P. Speed : 
\newblock Symmetric Wiener-Hopf factorisations in Markov Additive Processes.pdf
\newblock {\it Z.W.}, 26, 105-118, 1973.

\bibitem{Asmussen} S. Asmussen: \newblock {\it Applied Probability {Q}ueues. 2nd Edition.}
\newblock Springer, 2003.

\bibitem{Asm-rp1}
S.~Asmussen and H. Albrecher :
\newblock \emph{Ruin probabilities}, volume~14 of \emph{Advanced Series on
  Statistical Science \& Applied Probability}.
\newblock World Scientific Publishing Co. Pte. Ltd., Singapore, 2010.

\bibitem{Asm-apq2}
S.~Asmussen :
\newblock \emph{Applied probability and queues}, volume~51 of
  \emph{Applications of Mathematics (New York)}.
\newblock Springer-Verlag, New York, second edition, 2003.



\bibitem{BertoinLP}
J.~Bertoin :
\newblock \emph{L\'evy processes}, volume 121 of \emph{Cambridge Tracts in
  Mathematics}.
\newblock Cambridge University Press, Cambridge, 1996.

\bibitem{BD1994}
J. Bertoin and R. A. Doney :
\newblock Cram\'er's estimate for L\'evy processes. 
\newblock  {\it Statist. Probab. Lett.} {\bf 21} 363-365, 1994.

\bibitem{Bingham}
N. H. Bingham :
\newblock {Fluctuation Theory in Continuous Time.}
\newblock {\it Adv.  Appl. Prob.}, {\bf 7}, 705-766, 1975.


\bibitem{BGR} R. M. Blumenthal, R. K. Getoor and D. B. Ray :
\newblock  On the distribution of first hits for the symmetric stable process. 
\newblock {\it Trans. Amer. Math. Soc.} {\bf 99}, 540-554, 1961.

\bibitem{BZ}
T. Bogdan and T. Zak  : 
\newblock {On Kelvin Transformation}.
\newblock {\it J.  Theor. Probab.}, {\bf 19}, 89-120, 2006.

\bibitem{BG-mppt}
R.~M. Blumenthal and R.~K. Getoor :
\newblock {Markov processes and potential theory}.
\newblock Pure and Applied Mathematics, Vol. 29. Academic Press, New York,
  1968.


\bibitem{Cinlar1} E. \c{C}inlar : 
\newblock  Markov additive processes II
\newblock {\it Z.W.}, {\bf 24}, 95- 121, 1972.

\bibitem{Cinlar2} E. \c{C}inlar :
\newblock Levy systems of Markov additive processes
\newblock {\it Z.W.}, {\bf 31}, 175- 185, 1975.

\bibitem{CPR}
L.~Chaumont, H.~Pant{\'{\i}}, and V.~Rivero :
\newblock The {L}amperti representation of real-valued self-similar {M}arkov
  processes.
\newblock \emph{Bernoulli}, {\bf 19},  2494--2523, 2013.

\bibitem{Chy-Lam}
O.~Chybiryakov :
\newblock The {L}amperti correspondence extended to {L}\'evy processes and
  semi-stable {M}arkov processes in locally compact groups.
\newblock \emph{Stochastic Process. Appl.}, {\bf 116}, 
  857--872, 2006.


\bibitem{P_0}
S. Dereich, L. D\"oring and A.~E. Kyprianou :
\newblock Self-similar Markov processes started from the origin.
\newblock {\it Preprint.}


\bibitem{Iva-thesis}
J.~Ivanovs :
\newblock \emph{One-sided Markov additive processes and related exit problems}.
\newblock PhD thesis, Universiteit van Amsterdam, 2011.



\bibitem{kesten}  H. Kesten :   
\newblock Renewal Theory for Functionals of a Markov Chain with General State Space.
\newblock {\it    Ann. Probab.}
    {\bf 3}, 355-386, 1974.




\bibitem{Kaspi}
 H. Kaspi :
\newblock On the Symmetric Wiener-Hopf Factorization for Markov Additive Processes. 
\newblock {\it Z. Wahrsch. verw. Gebiete}, {\bf 59},  179-196, (1982).


\bibitem{KP} P. Klusik and  Z. Palmowski :
\newblock A Note onWiener--Hopf Factorization for Markov
Additive Processes.
\newblock {\it J. Theor. Probab.} {\bf 27}, 202-219, 2014.


\bibitem{KP} A. Kuznetsov and J.~C. Pardo :
\newblock Fluctuations of stable processes and exponential functionals of hypergeometric Levy processes
\newblock {\it  Acta Appl. Math.}, {\bf 123}, 113--139, 2013. 

\bibitem{T_0} 
A. Kuznetsov, A.~E. Kyprianou, J.~C. Pardo and A.~R. Watson :
\newblock The hitting time of zero for a stable process.
\newblock\emph{Electron. J. Probab.} {\bf 19},  1-26, 2014.





\bibitem{KPW} A.~E. Kyprianou, J.~C Pardo and A.~R. Watson :
\newblock  Hitting distributions of $\alpha$-stable processes via path censoring and self-similarity. 
\newblock {\it Ann.  Probab.}  {\bf 42}, 398-430, 2014.



\bibitem{Kyp}
A.~E. Kyprianou :
\newblock \emph{Introductory lectures on fluctuations of {L}\'evy processes
  with applications}.
\newblock Universitext. Springer-Verlag, Berlin, 2006.

\bibitem{Deep2}
A.~E. Kyprianou, V. Rivero and B. \c{S}eng\"{u}l :
\newblock \emph{Deep factorisation of the stable process II}.
\newblock {\it Working document.}





\bibitem{KRS}
A.~E. Kyprianou, V. Rivero and W. Satitkanitkul : \newblock \emph{Conditioned real self-similar Markov processes
}.
\newblock {\it Preprint.}

\bibitem{lalley} S. P. Lalley :
\newblock Conditional Markov renewal theory I. Finite and denumerable state space.
\newblock {\it Ann. Probab.}, {\bf 12}, 1113-1148, 1984.


\bibitem{R1} M. Reisz :
\newblock Int\'egrales de Riemann-Liouville et potentiels.
\newblock {\it Acta. Sci. Math. Szeged.} {\bf 9},
1-42, 1938.

\bibitem{R2} M. Reisz :
\newblock Rectification au travail ``Int\'egrales de Riemann-Liouville et potentiels". \newblock {\it Acta Sci. Math. Szeged.} {\bf 9}, 116-118, 1938.


\bibitem{Rog} B. A. Rogozin :
\newblock Distribution of the position of hit for stable and asymptotically stable random walks on an interval. 
\newblock {\it Teor. Verojatnost. i Primenen.} {\bf 17}, 342-349, 1972.

\bibitem{Sato}
K.~Sato :
\newblock \emph{L\'evy processes and infinitely divisible distributions},
  volume~68 of \emph{Cambridge Studies in Advanced Mathematics}.
\newblock Cambridge University Press, Cambridge, 1999.
\end{thebibliography}
\end{document}